%% file: numconf-draft9.tex
\documentclass[11pt]{article}
\usepackage[english]{babel}
\RequirePackage{amsmath}
\RequirePackage{amsthm}
\usepackage{amsfonts}
\usepackage{amssymb}
\usepackage{graphicx}

\usepackage{cite}
\usepackage{url}
\usepackage{paralist}
\usepackage[]{algorithm}

\usepackage[normalem]{ulem}
\usepackage{soul}
\usepackage{color}

\newtheorem{theorem}{Theorem}[section]
\newtheorem{lemma}[theorem]{Lemma}

\newtheorem{example}[theorem]{Example}

\newtheorem{remark}[theorem]{Remark}

\newcommand{\eop}{}

\newcommand{\Res}{\mathop{\mathrm{Res}}}
\renewcommand{\Re}{\mathop{\mathrm{Re}}}
\renewcommand{\Im}{\mathop{\mathrm{Im}}}
\renewcommand{\i}{\mathrm{i}}

\newcommand{\CC}{{\mathbb C}}

\newcommand{\LL}{{\mathcal L}}
\newcommand{\KK}{{\mathcal K}}

\newcommand{\bI}{{\bf I}}
\newcommand{\bM}{{\bf M}}
\newcommand{\bN}{{\bf N}}
\newcommand{\bz}{{\bf z}}

\newcommand{\bv}{{\bf v}}
\newcommand{\bx}{{\bf x}}
\newcommand{\by}{{\bf y}}
\newcommand{\bvb}{{\bf b}}
\newcommand{\bc}{{\bf c}}
\newcommand{\bt}{{\bf t}}

\include{def}

\newcommand{\coloneq}{\mathrel{\mathop:}=}

\begin{document}

\title{Numerical computation of the conformal map onto lemniscatic domains}
\author{Mohamed M. S. Nasser\footnotemark[2] \and J\"org Liesen\footnotemark[3] 
\and Olivier S\`{e}te\footnotemark[3]}

\date{\today}
\maketitle

\renewcommand{\thefootnote}{\fnsymbol{footnote}}

\footnotetext[2]{
Department of Mathematics, Statistics and Physics,
College of Arts and Sciences,
Qatar University,\\
P.O. Box: 2713, Doha, Qatar \\
\texttt{mms.nasser@qu.edu.qa}
}
\footnotetext[3]{Institute of Mathematics, Technische Universit\"{a}t Berlin, 
MA 4-5\\
Stra{\ss}e des 17. Juni 136\\
10623 Berlin, Germany\\
\texttt{\{liesen,sete\}@math.tu-berlin.de}
}

\renewcommand{\thefootnote}{\arabic{footnote}}

\begin{abstract}
We present a numerical method for the computation of the conformal
map from unbounded multiply-connected domains onto lemniscatic domains.
For $\ell$-times connected domains the method requires solving $\ell$ boundary 
integral equations with the Neumann kernel. This can be done in $O(\ell^2 n 
\log n)$ operations, where $n$ is the number of nodes in the discretization of 
each boundary component of the multiply connected domain. As demonstrated by 
numerical examples, the method works for domains with close-to-touching 
boundaries, non-convex boundaries, piecewise smooth boundaries, and for 
domains of high connectivity.
\end{abstract}

\textbf{Keywords} numerical conformal mapping; multiply connected domains;
lemniscatic domains; boundary integral equations; Neumann kernel.

\textbf{Mathematics Subject Classification (2010)}
30C30; 
45B05; 
65E05 

\section{Introduction}

In the theory of conformal mapping for multiply connected domains (open and 
connected sets) in the
extended complex plane $\widehat{\C} = \C \cup \{ \infty \}$, there are several
canonical domains onto which a given domain may be mapped.  The most commonly
considered canonical domains are slit domains (see, e.g., Chapter~VIII in the 
book of Nehari~\cite{Neh1952}),
circular domains, and domains with polygonal boundary.
Conformal maps onto these domains have been intensively studied, and several
numerical methods for the computation of these maps have been proposed.
Slit domains are considered, e.g., 
in~\cite{And-Mcn2012,Cro-Mar2006,Del-DEP2008,Nas-cmft09,Nas-siam09,Nas-jmaa11
,Nas-jmaa13}, circular domains in~\cite{Del2006,Luo2010,Nas-cmft15}, and
Schwarz--Christoffel maps for polygonal domains
in~\cite{Cro2005,Cro2007,Del2006,Del-DEP2006,Del-EKP2013,Del-EP2004,Del-Kro2011}
.

In this article we consider the numerical computation of the conformal map
onto \emph{lemniscatic domains}, which are another type of canonical domain.
A lemniscatic domain is a domain of the form
\begin{equation}
\cL \coloneq \Big\{ w \in \widehat{\C} :
\prod_{j=1}^\ell \abs{w-a_j}^{m_j} > \tau \Big\}, \label{eqn:lem_dom}
\end{equation}
where $a_1,\ldots,a_\ell \in \C$ are pairwise distinct,
$m_1, \ldots, m_\ell, \tau > 0$ are real numbers, and the exponents satisfy
\begin{equation}\label{eq:mj}
\sum_{j=1}^\ell m_j = 1.
\end{equation}
Lemniscatic domains where introduced by Walsh~\cite{Wal56}, who proved that if
$\cK$ is an $\ell$-times connected domain with $\infty \in \cK$, there exists a
lemniscatic domain $\cL$ as in~\eqref{eqn:lem_dom} and a conformal map $\Phi :
\cK \to \cL$ normalized by $\Phi(\infty) = \infty$ and $\Phi'(\infty) = 1$; see
Theorem~\ref{thm:existence_lem_dom} below for the precise statement.
The conformal map onto a lemniscatic domain is a direct generalization of the
Riemann map for simply connected domains, for if $\ell = 1$
in~\eqref{eqn:lem_dom}, then $\cL$ is is the exterior of a disk.

In addition to Walsh~\cite{Wal56}, the existence of the conformal map onto a
lemniscatic domain was shown by Grunsky~\cite{Grunsky1957,Grunsky1957a},
Jenkins~\cite{Jenkins1958} and Landau~\cite{Lan1961}. The last paper also
contains an iteration method for computing $\Phi$, which, however, requires
knowledge of the harmonic measure of parts of the boundary of the original
domain $\cK$. Recently, two of the present authors have investigated properties
of this map and constructed some explicit examples in~\cite{Set-Lie15}.

A remarkable feature of Walsh's conformal map is that it allows a direct 
generalization of the classical Faber polynomials, which are defined for simply 
connected compact sets, to compact sets with several components. The resulting
\emph{Faber--Walsh polynomials}, introduced by Walsh in~\cite{Wal58},
are likely to prove useful, given the vast number of both theoretical
and practical applications of the classical Faber polynomials. 
For further details on Faber--Walsh polynomials we refer to the recent 
paper~\cite{Set-Lie15_fwprop}.

Since the construction of conformal maps onto lemniscatic domains is in general
nontrivial and only a few explicit examples are known, it is desirable
to have a method for \textit{numerically computing} such maps.
In this paper we derive such a method and study it numerically. More precisely,
given an $\ell$-times connected domain $\cK$ with a sufficiently smooth boundary
our method computes the parameters defining the corresponding lemniscatic
domain $\cL$ as well as the boundary values of the conformal map
$\Phi : \cK \to \cL$. The method can be considered an extension of the
approach described
in~\cite{Nas-cmft09,Nas-siam09,Nas-jmaa11,Nas-jmaa13}
for the computation, in a unified way, of conformal maps onto all $39$ slit
domains identified by Koebe in~\cite{Koe1916}. The method described
in~\cite{Nas-cmft09,Nas-siam09,Nas-jmaa11,Nas-jmaa13} requires solving
a boundary integral equation with the generalized Neumann kernel.
Using the Fast Multipole Method (FMM), the integral equation for multiply
connected domains of connectivity $\ell$ can be solved numerically in
$O(\ell n \log n)$ operations where $n$ is the number of nodes in the
discretization of each boundary component~\cite{Nas-fast,Nas-siam13}.
The method presented in this article requires solving $\ell$ boundary integral
equations followed by solving a system of $\ell n+\ell$ non-linear equations.
The values of the conformal map $\Phi$ for interior points can then be 
calculated
using Cauchy's integral formula.

In recent years, several numerical methods have been proposed for
computing the conformal map of multiply connected domains onto different
types of canonical domains; 
see~\cite{And-Mcn2012,Cro2005,Cro2007,Cro-Mar2006,Del2006,Del-DEP2006,
Del-DEP2008,Del-EKP2013,Del-EP2004,Del-Kro2011,Luo2010,Nas-cmft09,Nas-siam09,
Nas-jmaa11,Nas-jmaa13,Nas-cmft15,Nas-siam13,Weg05,Yun-proc14} and the references 
cited therein.
However, most of these numerical methods are limited to certain types of
canonical domains or original domains. In comparison, the approach using
the boundary integral equation with the generalized Neumann kernel can
be used for a wide range of canonical domains. Moreover, it has been
successfully applied to domains of very high connectivity, with piecewise
smooth boundaries, with close-to-touching boundaries, and with complex
geometry; see~\cite{Nas-siam13,Nas-cmft15,Nas-fast,Nas-Sak}.

This paper is organized as follows. In Section~\ref{sect:dom} we state Walsh's
existence theorem for the conformal map onto lemniscatic domains. 
We then give the definition of the Neumann kernel in
Section~\ref{sect:neumann_kernel}.  In Section~\ref{sect:boundary_vals} we
derive equations for the boundary values of the conformal
map $\Phi$ and for the parameters of the lemniscatic domain $\cL$.
In Section~\ref{sect:numerics} we use these equations for the derivation of a 
numerical method for computing $\Phi$ and $\cL$.
Numerical examples with five different domains are presented in 
Section~\ref{sect:examples}. 
Concluding remarks are given in Section~\ref{sect:concl}.

\section{The conformal map onto lemniscatic domains}
\label{sect:dom}

The following result is Walsh's existence theorem from~\cite{Wal56} (see
also~\cite[Theorem~4]{Lan1961}), which shows that lemniscatic domains are
canonical domains for certain $\ell$-times connected domains.

\begin{theorem} \label{thm:existence_lem_dom}
Let $\cK$ be an unbounded domain in $\widehat{\C}$ with $\infty \in \cK$, and
let $\Gamma = \partial \cK$ consist of $\ell$ closed Jordan curves $\Gamma_1,
\ldots, \Gamma_\ell$.  Then there exist a uniquely determined lemniscatic domain
$\cL$ of the form~\eqref{eqn:lem_dom} and a uniquely determined bijective and
conformal map
\begin{equation}\label{eq:norm}
\Phi : \cK \to \cL \quad \text{ with } \quad
\Phi(z)=z+O\left(\frac{1}{z}\right) \; \text{ near infinity.}
\end{equation}
Further $\Phi$ extends to a continuous bijective map from $\overline{\cK} = \cK
\cup \Gamma$ to $\overline{\cL}$, and for each $j = 1, \ldots, \ell$, the image
of $\Gamma_j$ under $\Phi$ contains the point $a_j$ in its interior.

The number $\tau > 0$ is the transfinite diameter (i.e., the logarithmic
capacity) of the compact set $\widehat{\C} \backslash \cK$.
\end{theorem}

The uniqueness of the lemniscatic domain and the conformal map in 
Theorem~\ref{thm:existence_lem_dom} is forced by the normalization condition of 
$\Phi$ near infinity expressed in~\eqref{eq:norm}.
Note that if $c\in\C$ is any constant, then $\Phi_c=\Phi+c$ maps
$\cK$ bijectively and conformally onto the translated lemniscatic
domain $\cL+c$, and $\Phi_c$ satisfies the normalization conditions
$\Phi_c(\infty) = \infty$ and $\Phi_c'(\infty) = 1$.

Theorem~\ref{thm:existence_lem_dom} is illustrated in Figure~\ref{fig:3bw} for
a domain $\cK$ bounded by three Jordan curves; see Example~\ref{ex:3bw}
in Section~\ref{sect:examples} for details.

\begin{figure}
\centerline{
\includegraphics[width=0.5\textwidth]{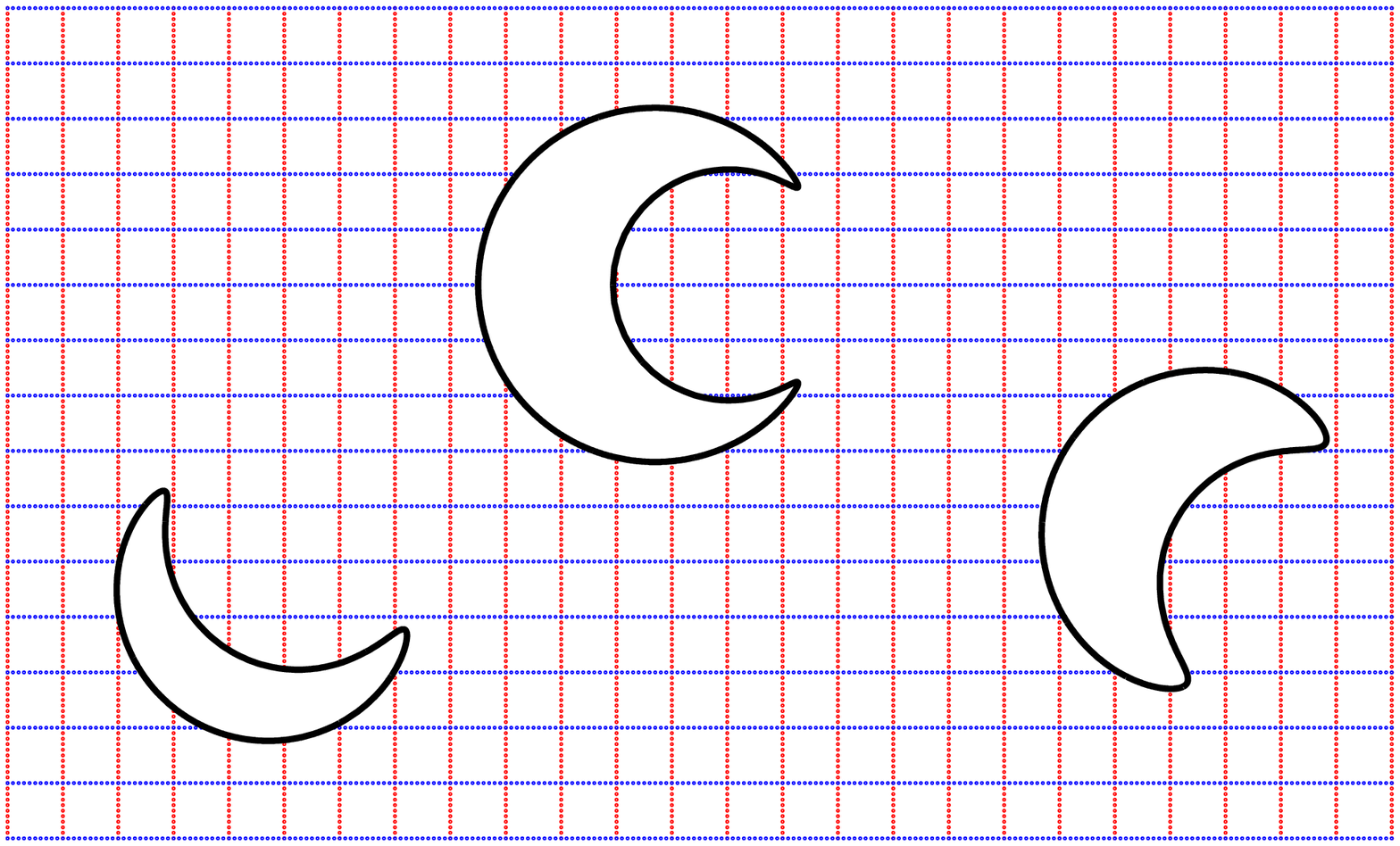}
\includegraphics[width=0.5\textwidth]{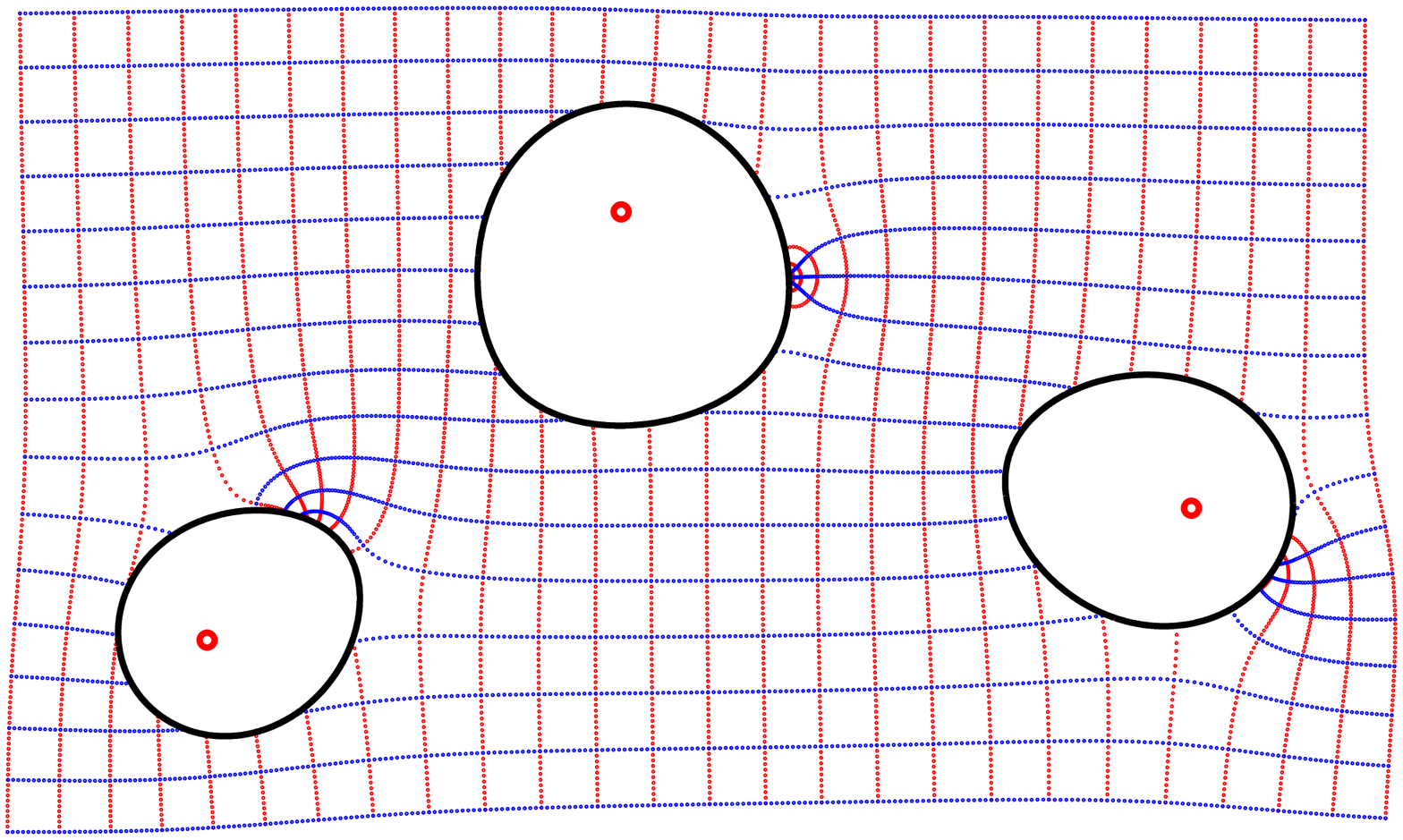}
}
\caption{Illustration of Theorem~\ref{thm:existence_lem_dom}.  Original domain
(left) and corresponding lemniscatic domain (right).  The red dots are the
centers $a_1, a_2, a_3$.}
\label{fig:3bw}
\end{figure}

We will need the following lemma.

\begin{lemma}\label{lem:Upne0}
In the notation of Theorem~\ref{thm:existence_lem_dom}, let
\begin{equation*}
U(w) = \prod_{j=1}^\ell (w-a_j)^{m_j},
\end{equation*}
so that $\cL = \{ w \in \widehat{\C} : \abs{U(w)} > \tau \}$.  Then $U'(w)
\neq 0$ for $w \in \partial \cL = \{ w \in \widehat{\C} : \abs{U(w)} =
\tau \}$.
\end{lemma}

\begin{proof}
We show that the zeros of $U'$ are the critical points of a Green's function
with pole at infinity, and that these are not located on $\partial \cL$.

The function $U$ is analytic but multi-valued in $\C \backslash \{ a_1, \ldots,
a_\ell \}$.  Let $0 < \rho < \tau$ and consider the auxiliary function
\begin{equation*}
F(w) = \log U(w) - \log(\rho),
\end{equation*}
which is analytic but not single-valued in $\C \backslash \{ a_1, \ldots,
a_\ell \}$.  Then the harmonic functions
\begin{equation*}
G(w) = \re(F(w)) =  \log \abs{U(w)} - \log(\rho) \quad \text{ and } \quad H(w)
= \im(F(w))
\end{equation*}
are conjugates, i.e., $G_x = H_y$ and $G_y = - H_x$.  The derivative of $F$ is
given by
\begin{equation*}
\frac{U'}{U} = F' = \frac{1}{2} (F_x - i F_y)
= \frac{1}{2} ( G_x + i H_x - i (G_y + i H_y) ) = G_x - i G_y.
\end{equation*}
Since $G$ is real-valued, $U'(w_0) = 0$ if and only if $G_x(w_0) = 0 =
G_y(w_0)$, i.e., the zeros of $U'$ are the critical points of $G$.

The function $G$ is Green's function with pole at infinity for the $\ell$-times
connected domain $\{ w \in \widehat{\C} : \abs{U(w)} > \rho \}$,
see~\cite[p.~28]{Wal58}.  Its contour lines (level curves)
\begin{equation*}
\Lambda_\sigma = \{ w \in \C : G(w) = \log(\sigma) \}
= \{ w \in \C : \abs{U(w)} = \sigma \rho \}, \quad \sigma \geq 1,
\end{equation*}
have the following properties~\cite[p.~67]{Walsh69}:
\begin{enumerate}
\item For sufficiently small $\sigma > 1$, $\Lambda_\sigma$ consists of $\ell$
contours, each surrounding exactly one of the boundary contours (of $\{ w \in
\widehat{\C} : \abs{U(w)} > \rho \}$).

\item $\Lambda_\sigma$ grows with $\sigma$, in the sense that if $\sigma <
\sigma'$, then $\Lambda_\sigma$ is contained in the interior of
$\Lambda_{\sigma'}$.

\item If $\sigma$ increases and $\Lambda_\sigma$ crosses through an $m$-fold
critical point of $G$, the number of components of $\Lambda_\sigma$ decreases
by exactly $m$.
\end{enumerate}
Now, since $\cL$ is $\ell$-times connected, $\partial \cL =
\Lambda_{\tau/\rho}$ has $\ell$ components.  Therefore, no critical points of
$G$ can lie on $\partial \cL$ (or interior to $\partial \cL$), which finishes
the proof.
\eop
\end{proof}

\section{The Neumann kernel}
\label{sect:neumann_kernel}

From now on we assume that $\cK$ is a given domain as in 
Theorem~\ref{thm:existence_lem_dom} which additionally has a \emph{sufficiently 
smooth} boundary $\Gamma$, oriented such that $\cK$ is on the left of $\Gamma$. 
More precisely, we assume that each boundary
curve $\Gamma_j$ is parameterized by a $2\pi$-periodic twice continuously
differentiable complex function $\eta_j (t)$ with non-vanishing first
derivative
$\dot\eta_j(t)=d\eta_j(t)/dt\ne0$ for $t\in J_j=[0,2\pi]$, $j=1,2,\ldots,\ell$.
(A dot always denotes the derivative with respect to the parameter $t$.)
The total parameter domain $J$ is the disjoint union of the $\ell$ intervals
$J_1,\ldots,J_\ell$,
\begin{equation*}
J \coloneq \bigsqcup_{j=1}^\ell J_j=\bigcup_{j=1}^\ell \{(t,j)\;:\;t\in J_j\},
\end{equation*}
i.e., the elements of $J$ are ordered pairs $(t,j)$ where $j$ is an auxiliary
index indicating which of the intervals contains the point $t$. We define a
parametrization of the whole boundary $\Gamma$ as the complex function $\eta$
defined on $J$ by
\begin{equation}\label{e:eta-1}
\eta(t,j) \coloneq \eta_j(t), \quad t\in J_j,\quad j=1,\ldots,\ell.
\end{equation}
We shall assume for a given $t$ that the auxiliary index $j$ is known, so we
replace the pair $(t,j)$ in the left-hand side of~\eqref{e:eta-1} by $t$.
Thus, the function $\eta$ in~\eqref{e:eta-1} is written as
\begin{equation*}
\eta(t)= \left\{ \begin{array}{l@{\hspace{0.5cm}}l}
\eta_1(t),&t\in J_1,\\
\hspace{0.3cm}\vdots\\
\eta_\ell(t),&t\in J_\ell. \\
\end{array}
\right.
\end{equation*}

Let $H$ be the space of all real-valued H\"older continuous functions on the
boundary $\Gamma$.  In view of the smoothness of the parametrization $\eta(t)$
of the boundary $\Gamma$, any function $\phi\in H$ can be  interpreted via
$\hat\phi(t)\coloneq\phi(\eta(t))$ as a $2\pi$-periodic H\"older continuous 
function
of the parameter $t$ on $J$, and vice versa. Henceforth, in this paper, we
shall not distinguish between $\phi(t)$ and $\phi(\eta(t))$.

We define the Neumann kernel $N(s,t)$ for $(s,t)\in J\times J$ by
\begin{equation*}
N(s,t) \coloneq 
\frac{1}{\pi}\Im\left(\frac{\dot\eta(t)}{\eta(t)-\eta(s)}\right).
\end{equation*}
It is a particular case of the generalized Neumann kernel considered 
in~\cite{Weg-Nas}. 
Similarly, the kernel $M(s,t)$ defined for $(s,t)\in J\times J$ by
\begin{equation*}
M(s,t) \coloneq 
\frac{1}{\pi}\Re\left(\frac{\dot\eta(t)}{\eta(t)-\eta(s)}\right),
\end{equation*}
is a particular case of the kernel $M$ considered in~\cite{Weg-Nas}.
The above kernels have been used in~\cite{Nas-siam09,Nas-jmaa11}
for computing the conformal map from unbounded multiply domains onto the
canonical slit domains; see also~\cite{Nas-siam13,Nas-fast,Nas-amc11}.
The Neumann kernel also appears frequently in the integral equations for
potential theory; see, e.g.,~\cite{Atk97,Gre93,Hen3,Kre14,Weg05}.

\begin{lemma}[see~\cite{Weg-Nas}]
(a) The kernel $N$ is continuous with
\begin{equation*}
N(t,t)= \frac{1}{2\pi} \Im\frac{\ddot\eta(t)}{\dot\eta(t)}.
\end{equation*}
(b) When $s,t\in J_j$  are in the same parameter interval $J_j$, then
\begin{equation*}
M(s,t)= -\frac{1}{2\pi} \cot \frac{s-t}{2} + M_1(s,t)
\end{equation*}
with a continuous kernel $M_1$ which takes on the diagonal the values
\begin{equation*}
M_1(t,t)= \frac{1}{2\pi} \Re\frac{\ddot\eta(t)}{\dot\eta(t)}.
\end{equation*}
\end{lemma}

We define integral operators $\bN$ and $\bM$ on the space $H$ by
\begin{align*}
\bN\mu(s) &\coloneq \int_J N(s,t) \mu(t) dt, \quad s\in J, \\
\bM\mu(s) &\coloneq \int_J  M(s,t) \mu(t) dt, \quad s\in J.
\end{align*}
The integral operator $\bN$ is a compact operator and the operator $\bM$ is a
singular operator. Both operators $\bN$ and $\bM$ are bounded on the space $H$
and map $H$ into itself~\cite{Weg-Nas}.
Finally, any piecewise constant function $\nu\in H$ defined by
\begin{equation*}
\nu(t)=\nu_j\quad {\rm for}\quad t\in J_j,
\end{equation*}
with real constants $\nu_j$ for $j=1,\ldots,\ell$, will be denoted by
\begin{equation*}
\nu(t)=(\nu_1,\ldots,\nu_\ell),\quad t\in J.
\end{equation*}

\section{The boundary values of $\Phi$ and the parameters of $\cL$}
\label{sect:boundary_vals}

In this section we derive equations for the boundary values of the conformal 
map $\Phi$ and for the parameters of the lemniscatic domain $\cL$.  These 
equations will be used for the numerical computation of $\Phi$ and $\cL$ in 
Section~\ref{sect:numerics} below.

In the notation of Theorem~\ref{thm:existence_lem_dom}, the function 
$\Phi$ maps the boundary $\Gamma$ of $\KK$ onto the boundary of the lemniscatic 
domain $\LL$, i.e., for $t\in J$,
\begin{equation*}
\prod_{j=1}^\ell \abs{ \Phi(\eta(t))-a_j }^{m_j} = \tau,
\end{equation*}
or, equivalently,
\begin{equation}\label{eq:bd-0}
\sum_{j=1}^\ell m_j \log \abs{ \Phi(\eta(t))-a_j } = \log\tau.
\end{equation}
To determine the parameters of the lemniscatic domain and the boundary values 
of $\Phi$, we fix for each $j = 1, \ldots, \ell$ an auxiliary point $\alpha_j$ 
in the interior of the curve $\Gamma_j$, and define the functions
\begin{equation}\label{eq:gam-j}
\gamma_j(t) \coloneq -\log\abs{\eta(t)-\alpha_j}, \quad t \in J, \quad j = 1, 
\ldots, \ell.
\end{equation}

The first equation for the boundary values of $\Phi$ will be obtained from the 
boundary values of the function
\begin{equation}\label{eq:F(z)}
f(z)=\sum_{j=1}^\ell m_j\log\left(\frac{\Phi(z)-a_j}{z-\alpha_j}\right),
\end{equation}
where the branch of the logarithm with $\log(1) = 0$ is chosen.  The function
$f$ is analytic in the domain $\KK$ with $f(\infty)=0$.
By~\eqref{eq:bd-0}, its boundary values are given by
\begin{equation}\label{eq:F-bd}
f(\eta(t)) = \log\tau+\gamma(t)+\i\mu(t), \quad t \in J,
\end{equation}
where
\begin{align}
\label{eq:gam}
\gamma(t) &\coloneq \sum_{j=1}^\ell m_j \gamma_j(t)
= - \sum_{j=1}^\ell m_j \log \abs{\eta(t)-\alpha_j}, \\
\mu(t) &\coloneq \Im f(\eta(t)). \nonumber
\end{align}
Although the functions $\gamma_j$ are known from~\eqref{eq:gam-j}, the constant 
$\log\tau$ and the functions $\gamma$ and $\mu$ are not known a priori.
We now show how the boundary values~\eqref{eq:F-bd} of $f$ can be computed
from the functions $\gamma_j$.  A key ingredient is the following theorem
from~\cite{Nas-siam09}.

\begin{theorem}\label{T:ie}
For each function $\gamma_j$ from~\eqref{eq:gam-j}, there exists a unique
real-valued function $\mu_j$ and a unique piecewise constant real-valued
function $h_j=(h_{1,j},\ldots,h_{\ell,j})$ such that
\begin{equation}\label{eq:F-j}
f_j(\eta(t))=\gamma_j(t)+h_j(t)+\i\mu_j(t), \quad t \in J,
\end{equation}
are boundary values of an analytic function $f_j$ in $\cK$ with
$f_j(\infty)=0$.
The function $\mu_j$ is the unique solution of the integral equation
\begin{equation}\label{eq:ie-j}
(\bI-\bN)\mu_j=-\bM\gamma_j
\end{equation}
and the function $h_j$ is given by
\begin{equation}\label{eq:h-j}
h_j=[\bM\mu_j-(\bI-\bN)\gamma_j]/2.
\end{equation}
\end{theorem}

Theorem~\ref{T:ie} allows to compute the functions $h_j$ and $\mu_j$ from the
known function $\gamma_j$, so that the boundary values of $f_j$ are known.
The following theorem shows that $f = \sum_{j=1}^\ell m_j f_j$ holds, by
relating $\mu$ and $\tau$ to the known functions $h_j$ and $\mu_j$.

\begin{theorem} 
Let $f$ be the function from~\eqref{eq:F(z)} and let the notation be as in
Theorem~\ref{T:ie}.  Then the function $\mu$ and the constant $\log\tau$
in~\eqref{eq:F-bd} are given by
\begin{align}
\label{eq:mu-mu-j}
\mu(t)&= \sum_{j=1}^{\ell}m_j\mu_j(t),\\
\label{eq:r-h}
\log\tau&= \sum_{j=1}^{\ell}m_jh_j(t),
\end{align}
and we have $f(z) = \sum_{j=1}^\ell m_j f_j(z)$ in $\cK$.
\end{theorem}

\begin{proof}
Define the auxiliary function $g$ in $\cK$ by
\begin{equation*}
g(z) = f(z) - \sum_{j=1}^\ell m_j f_j(z).
\end{equation*}
Then $g$ is analytic in $\cK$ with $g(\infty) = 0$, since each of the functions
$f$ and $f_j$, $j = 1, 2, \ldots, \ell$, has these properties.  In view
of~\eqref{eq:F-bd} and~\eqref{eq:F-j} the boundary values of $g$ are given by
\begin{equation*}
g(\eta(t)) = \log\tau + \gamma(t) + \i \mu(t) - \sum_{j=1}^\ell m_j \gamma_j(t)
- \sum_{j=1}^\ell m_j h_j(t) - \i \sum_{j=1}^\ell m_j \mu_j(t),
\end{equation*}
which, by~\eqref{eq:gam}, simplifies to
\begin{equation}\label{eq:g-bd}
g(\eta(t))=
\log\tau-\sum_{j=1}^{\ell} m_j h_j(t) + \i \bigg(\mu(t) -\sum_{j=1}^{\ell}
m_j \mu_j(t) \bigg).
\end{equation}
Thus,
\begin{equation*}
\Re[g(\eta(t))]= \log\tau-\sum_{j=1}^{\ell}m_jh_j(t),
\end{equation*}
i.e., the real part of the boundary values of the analytic function $g$ is a
piecewise constant function. Since $g(\infty)=0$, then $g$ is the zero 
function~\cite[p.~165]{Mus77}.  Hence,~\eqref{eq:r-h}
and~\eqref{eq:mu-mu-j} follow from~\eqref{eq:g-bd}.
\eop
\end{proof}

In order to compute the boundary values~\eqref{eq:F-bd} of $f$, it remains to 
compute
the numbers $m_1, m_2, \ldots, m_\ell$.

\begin{theorem}\label{thm:lin_sys_mv_tau}
Let the functions $h_j=(h_{1,j},\ldots,h_{\ell,j})$ be as in 
Theorem~\ref{T:ie}.
The unknown $\ell+1$ real constants $m_1,\ldots,m_\ell, \log\tau$ are the
unique solution of the linear system
\begin{equation}\label{eq:pj-sys}
A \begin{bmatrix}
m_1 \\
m_2 \\
\vdots \\
m_\ell \\
\log\tau \\
\end{bmatrix}
=
\begin{bmatrix}
0 \\
0 \\
\vdots \\
0 \\
1 \\
\end{bmatrix},
\quad \text{ where } A =
\begin{bmatrix}
h_{1,1}  &h_{1,2}  &\cdots  &h_{1,\ell}  &-1 \\
h_{2,1}  &h_{2,2}  &\cdots  &h_{2,\ell}  &-1 \\
\vdots  &\vdots  &\ddots  &\vdots  &\vdots \\
h_{\ell,1}  &h_{\ell,2}  &\cdots  &h_{\ell,\ell}  &-1 \\
1       &1       &\cdots  &1       &0 \\
\end{bmatrix}.
\end{equation}
\end{theorem}

\begin{proof}
As shown above, the $\ell+1$ real constants $m_1, \ldots, m_\ell, \tau$
satisfy~\eqref{eq:r-h}.  Since the functions $h_1, \ldots, h_\ell$ are
piecewise constant, it is easy to see that~\eqref{eq:mj} and~\eqref{eq:r-h} can
be written in the form of the linear algebraic system~\eqref{eq:pj-sys}.

We next show that $A$ is nonsingular.  Suppose that $[q_1,\ldots,q_\ell,c]^T$
is a (real) solution of the homogeneous linear system
\begin{equation}\label{eq:pj-sysh}
A
\begin{bmatrix}
q_1 \\
\vdots \\
q_\ell \\
c \\
\end{bmatrix}
= 0.
\end{equation}
Then we have
\begin{subequations}
\renewcommand{\theequation}{\theparentequation\alph{equation}}
\begin{align}
\label{eq:qj-1}
\sum_{j=1}^\ell q_j h_j(t) &= c, \\
\label{eq:qj-2}
\sum_{j=1}^\ell q_j &= 0.
\end{align}
\end{subequations}
Let the auxiliary function $\tilde f(z)$ be defined by
\begin{equation*}
\tilde f(z)=\sum_{j=1}^\ell q_j f_j(z)
\end{equation*}
where the functions $f_j$ are as in Theorem~\ref{T:ie}.
Then $\tilde f$ is analytic in $\cK$ with $\tilde f(\infty)=0$, and its 
boundary values are given by
\begin{equation*}
\tilde f(\eta(t)) = \sum_{j=1}^\ell q_j \gamma_j(t) + \sum_{j=1}^\ell q_j h_j(t)
+ \i\sum_{j=1}^\ell q_j \mu_j(t),
\end{equation*}
which by~\eqref{eq:gam-j} and~\eqref{eq:qj-1} can be written as
\begin{equation*}
\tilde
f(\eta(t)) = c - \sum_{j=1}^\ell q_j \log|\eta(t)-\alpha_j| + \i\sum_{j=1}^\ell
q_j \mu_j(t).
\end{equation*}
Define a function $\tilde g(z)$ on $\cK$ by
\begin{equation}\label{eq:hg}
\tilde g(z) = \tilde f(z) + \sum_{j=1}^\ell q_j \log(z-\alpha_j).
\end{equation}
The function $\tilde g(z)$ is analytic in $\cK$ but is not necessarily
single-valued.  For large $\abs{z}$, we have
\begin{equation*}
\log(z-\alpha_j) = \log\left(z\left(1-\frac{\alpha_j}{z}\right)\right)
=\log z+\log\left(1-\frac{\alpha_j}{z}\right),
\end{equation*}
which implies in view of~\eqref{eq:qj-2} that
\begin{equation*}
\sum_{j=1}^\ell q_j \log(z-\alpha_j)
= \log(z) \sum_{j=1}^\ell q_j + \sum_{j=1}^\ell q_j
\log\left(1-\frac{\alpha_j}{z}\right)
= \sum_{j=1}^\ell q_j \log\left(1-\frac{\alpha_j}{z}\right).
\end{equation*}
Hence, the second term in the right-hand side of~\eqref{eq:hg} vanishes at
$\infty$. Since $\tilde f(\infty)=0$, we have $\tilde g(\infty)=0$.
Let the real function $u$ be defined for $z\in \cK \cup \Gamma$ by
\begin{equation*}
u(z) = \Re\tilde g(z).
\end{equation*}
Then $u$ is harmonic in $\cK$, and satisfies the Dirichlet boundary condition
\begin{equation}\label{eq:u-Dir}
u=c \quad \text{on } \Gamma,
\end{equation}
i.e., $u$ is constant on the boundary $\Gamma$.  The Dirichlet
problem~\eqref{eq:u-Dir} has the unique solution $u(z)=c$ for all
$z\in \cK \cup \Gamma$, so that the real part of $\tilde g$ is constant.
Then $\tilde{g}$ is constant in $\cK$ by the Cauchy-Riemann equations, and
$\tilde{g}(\infty) = 0$ shows that $\tilde g(z)=0$ for all $z \in \cK$, and, in
particular, $c = 0$.  Then,~\eqref{eq:hg} implies that
\begin{equation*}
\tilde f(z) = - \sum_{j=1}^\ell q_j \log(z-\alpha_j) \quad \text{ for all } z
\in \cK,
\end{equation*}
which is impossible unless $q_1=q_2=\cdots=q_\ell = 0$ (since the function on
the left-hand side is single-valued and the function on the right-hand side is
multi-valued). Thus, the homogeneous linear system~\eqref{eq:pj-sysh} has only
the trivial solution $[q_1, \ldots, q_\ell, c]^T = 0$, and $A$ is nonsingular.
\eop
\end{proof}

By obtaining the real constants $m_1,\ldots,m_\ell$, we can compute the
functions $\gamma$ and $\mu$ from~\eqref{eq:gam} and~\eqref{eq:mu-mu-j}.
By~\eqref{eq:F-bd}, the boundary values of the analytic function $f$
from~\eqref{eq:F(z)} are given by
\begin{equation}\label{eq:log-an2}
f(\eta(t)) =
\sum_{j=1}^\ell m_j \log\left(\frac{\Phi(\eta(t))-a_j}{\eta(t)-\alpha_j}\right)
=\log\tau+\gamma(t)+\i\mu(t).
\end{equation}
This is the first equation needed to compute the boundary values of the
conformal map $\Phi$.  Note that in~\eqref{eq:log-an2} only $\Phi(\eta(t))$ and
the complex numbers $a_j$ are unknown.  We will need one more set of equations 
to determine these quantities.

\begin{lemma} 
The boundary values of the function $\Phi$ and the $\ell$ constants
$a_1,\ldots, a_\ell$ satisfy the $\ell$ equations
\begin{equation}\label{eq:log-int}
\frac{1}{2\pi\i}\int_{J}\log\left(\frac{\Phi(\eta(t))-a_j}{\eta(t)-\alpha_j}
\right)\dot\eta(t)dt +\alpha_j-a_j=0, \quad j=1,2,\ldots,\ell.
\end{equation}
\end{lemma}

\begin{proof}
Since the functions
\begin{equation}\label{eq:fj}
\psi_j(z)=\log\left(\frac{\Phi(z)-a_j}{z-\alpha_j}\right), \quad
j=1,2,\ldots,\ell,
\end{equation}
where the branch of the logarithm with $\log(1) = 0$ is chosen,
are analytic in the domain $\cK$ including the point at $\infty$ with
$\psi_j(\infty)=0$, we have
\begin{equation}\label{eq:fj-i1}
\frac{1}{2\pi\i} \int_\Gamma \psi_j(\eta)d\eta = \Res_{z=\infty}\psi_j(z)
= -\Res_{z=0}\frac{1}{z^2}\psi_j\left(\frac{1}{z}\right), \quad
j=1,2,\ldots,\ell;
\end{equation}
see~\cite[pp.~107-108]{Kap66}.  Let $\widehat{\cK}$ be the image of the domain
$\cK$ under the mapping $\frac{1}{z}$ and the functions $g_j(z)$ be defined on
$\widehat{\cK}$ by
\begin{equation*}
g_j(z) = \psi_j\left(\frac{1}{z}\right), \quad j=1,2,\ldots,\ell.
\end{equation*}
Then $g_j$ is analytic in $\widehat{\cK}$ with $g_j(0)=0$ for
$j=1,2,\ldots,\ell$. Hence
\begin{equation}\label{eq:Res-g}
\Res_{z=0}\frac{1}{z^2}\psi_j\left(\frac{1}{z}\right)
= \lim_{z\to0} \frac{1}{z}\psi_j\left(\frac{1}{z}\right)
= \lim_{z\to0} \frac{g_j(z)}{z}
= g'_j(0), \quad j=1,2,\ldots,\ell.
\end{equation}
To compute $g'_j(0)$, we have from~\eqref{eq:fj}
\begin{equation*}
g_j(z) =\log\left(\frac{\Phi(1/z)-a_j}{1/z-\alpha_j}\right)
=\log\left(\frac{z\Phi(1/z)-a_jz}{1-\alpha_jz}\right).
\end{equation*}
For small $\abs{z}$, the normalization~\eqref{eq:norm} of $\Phi$ 
implies $z \Phi(1/z) = 1 + O(z^2)$, and
\begin{equation*}
g_j(z) = \log \left( \frac{1 - a_j z + O(z^2)}{1 - \alpha_j z} \right),
\end{equation*}
so that
\begin{equation*}
g_j'(z) = \frac{-a_j + O(z)}{1 - a_j z + O(z^2)} - \frac{-\alpha_j}{1 - 
\alpha_j z},
\end{equation*}
which implies
\begin{equation}\label{eq:gj-0}
g_j'(0) = \alpha_j - a_j, \quad j = 1, 2, \ldots, \ell.
\end{equation}
The assertion of the lemma follows from~\eqref{eq:fj}, \eqref{eq:fj-i1},
\eqref{eq:Res-g}, and~\eqref{eq:gj-0}.
\eop
\end{proof}

\section{The numerical computation of the conformal \\ map $\Phi$ and
lemniscatic domain $\cL$}
\label{sect:numerics}

In this section we discuss the numerical aspects of the computation of the
conformal map $\Phi$ and the lemniscatic domain $\cL$ based on the results
from Section~\ref{sect:boundary_vals}.

We first consider the numerical solution of the boundary integral
equations from Theorem~\ref{T:ie}, and the computation of the parameters
$m_1, \ldots, m_\ell$ and $\log \tau$ from Theorem~\ref{thm:lin_sys_mv_tau}.
We then consider the computation of the boundary values of $\Phi$
and of $a_1, \ldots, a_\ell$ by solving the equations~\eqref{eq:log-an2}
and~\eqref{eq:log-int}.  Finally, we discuss the computation of the
values of $\Phi$ at interior points of $\cK$ by Cauchy's integral formula.

\subsection{Computation of the parameters $m_1,\ldots,m_\ell, \log\tau$}
\label{sect:compute_mv_tau}

The $\ell$ boundary integral equations~\eqref{eq:ie-j} can be solved accurately
by the Nystr\"om method with the trapezoidal rule~\cite{Atk97,Kre14}
(see~\cite{Nas-cmft09,Nas-siam09,Nas-fast,Nas-siam13} for more details). Let
$n$ be a given even positive integer. For $k=1,2,\ldots,\ell$, each interval
$J_k$ is discretized by the $n$ equidistant nodes
\[
s_{k,p}=(p-1)\frac{2\pi}{n}\in J_k, \quad p=1,2,\ldots,n.
\]
Hence, the total number of nodes in the parameter domain $J$ is $\ell n$.
We denote these nodes by $t_i$, $i=1,2,\ldots,\ell n$, i.e.,
\begin{equation}\label{e:ti-skp}
t_{(k-1)n+p}=s_{k,p}\in J, \quad k=1,2,\ldots,\ell, \quad p=1,2,\ldots,n.
\end{equation}
For domains with piecewise smooth boundaries, singularity
subtraction~\cite{Rat93} and the trapezoidal rule with a graded
mesh~\cite{Kre90} are used (see~\cite{Nas-fast}). By discretizing the integral
equation~\eqref{eq:ie-j} by the Nystr\"om method with the trapezoidal rule, we
obtain the $\ell n \times \ell n$ linear algebraic system 
$(I-B)\bx=\by$; see, e.g.,
\cite[equation~(59)]{Nas-siam13} for an explicit formula for this system.

Since the integral operator $\bN$ is compact, the only possible
accumulation point of its eigenvalues is~$0$; see, e.g.,~\cite[p.~40]{Kre14}.
Moreover, as shown in~[36, Corollary~2], the eigenvalues
of $\bI-\bN$ are contained in the interval $(0,2]$ 
(see also~\cite[Theorem~10.21]{Kre14}). Consequently,
for sufficiently large $n$, the eigenvalues of the discretized
operator $I-B$ are contained in $(0,2]$ and they cluster around $1$.
Numerical illustrations of this eigenvalue distribution are shown 
in~\cite[Figures~4--5]{Nas-amc11}.

We solve the discretized system $(I-B)\bx=\by$ 
using the (full) GMRES method~\cite{Saa-Sch}.
Each GMRES iteration requires one matrix-vector product with $I-B$.
This product can be efficiently computed using the Fast Multipole Method
in just $O(\ell n)$ operations~\cite{Gre-Gim12,Gre-Rok}. It was already
observed in~\cite{Nas-siam13}, that the number of GMRES iterations
for obtaining a very good approximation of the exact solution
(relative residual norm $<10^{-12}$) is virtually independent of
the given domain and the number of nodes in the
discretization of its boundary. In all numerical experiments we performed,
we found that very few steps of (full) GMRES 
reduce the relative residual norm to $10^{-12}$ or smaller. No
preconditioning was required. Several examples are given
in Section~\ref{sect:examples} below.
We believe that the very fast convergence of GMRES is due to the strong
clustering of the eigenvalues of $I-B$ around $1$ with only a few
``outliers'' and none of these ``outliers'' being close to zero.
A more detailed analysis of this situation is a subject of further work.

In the numerical examples shown in Section~\ref{sect:examples}
the MATLAB function \verb|fbie| from~\cite{Nas-fast} is used in
order to obtain approximations to the unique solution
$\mu_j$ of the integral equation~\eqref{eq:ie-j} and the function $h_j$
in~\eqref{eq:h-j}, respectively. Within \verb|fbie| we apply the MATLAB
function \verb|gmres| with the tolerance $10^{-14}$ for the relative
residual norm. The matrix-vector product with $I-B$
is computed using
the MATLAB function \verb|zfmm2dpart| from the MATLAB toolbox FMMLIB2D
developed by Greengard and Gimbutas~\cite{Gre-Gim12}.
We thus obtain approximations to the values of the functions $\mu_j$ and
$h_j$ for $j=1,\ldots,\ell$, at the points $t_i$ for $i=1,2,\ldots,\ell n$.
Then the values of the constants $h_{k,j}$ are approximated by
\begin{equation*}
h_{k,j} = \frac{1}{n}\sum_{i=1+(k-1)n}^{kn}h_j(t_i), \quad
j=1,2,\ldots,\ell, \quad k=1,2,\ldots,\ell.
\end{equation*}
Since the function \verb|fbie| requires $O(\ell n \log n)$ operations,
the computational cost for solving the $\ell$ integral equations~\eqref{eq:ie-j}
and computing the $\ell$ functions $h_j$ in~\eqref{eq:h-j} is $O(\ell^2 n \log
n)$ operations.
For more details, we refer the reader to~\cite{Gre-Gim12,Nas-fast,Nas-siam13}.

Next, the values of the parameters $m_1,\ldots,m_\ell,\log\tau$ are computed by
solving the linear algebraic system~\eqref{eq:pj-sys}.  Since $\ell$ usually is 
not large, this $(\ell+1) \times (\ell+1)$ system can be solved directly, e.g., 
using backslash in MATLAB.
Finally, the values of the functions $\gamma$ and $\mu$ at the points $t_i$ for
$i=1,2,\ldots,\ell n$ can be computed from~\eqref{eq:gam} 
and~\eqref{eq:mu-mu-j}.

\subsection{Computation of the $a_j$ and the boundary values of $\Phi$}
\label{sect:compute_aj_Phi}

In the preceding section, we have computed the parameters $m_1,\ldots,m_\ell,
\log\tau$ and the values of the functions $\mu$ and $h_j$ at the points $t_i$
for $i=1,2,\ldots,\ell n$.
In this section, we shall compute the values of $a_1,\ldots, a_\ell$
and the values of the function $\Phi(\eta(t))$ at the points $t_i$ for
$i=1,2,\ldots,\ell n$, by solving a non-linear system of equations.

\subsubsection{The non-linear system}

Let
\begin{equation} \label{eqn:bdry_vals}
w_i=\Phi(\eta(t_i)), \quad p_i=\log\tau+\gamma(t_i)+\i\mu(t_i), \quad 
i=1,2,\ldots,\ell n.
\end{equation}
The $p_i$ are known from Section~\ref{sect:compute_mv_tau}.
We will compute $w_1, w_2, \ldots, w_{\ell n}$ and $a_1, a_2, \ldots$, $a_\ell$.
We have from~\eqref{eq:log-an2} the following $\ell n$ non-linear
algebraic equations in the $\ell n+\ell$ unknowns $w_1,w_2,\ldots,w_{\ell n},
a_1,\ldots,a_\ell$,
\begin{subequations}\label{e:non}
\renewcommand{\theequation}{\theparentequation\alph{equation}}
\begin{equation}\label{eq:non-1}
\sum_{j=1}^{\ell}m_j\log\left(\frac{w_i-a_j}{\eta(t_i)-\alpha_j}\right)
=p_i, \quad i=1,2,\ldots,\ell n.
\end{equation}
By discretizing the integral in~\eqref{eq:log-int}, we also have the following 
$\ell$ non-linear algebraic equations in the $\ell n+\ell$ unknowns 
$w_1,w_2,\ldots,w_{\ell n}, a_1,\ldots,a_\ell$,
\begin{equation}\label{eq:non-2}
\frac{1}{n\i}\sum_{j=1}^{\ell
n}\log\left(\frac{w_j-a_i}{\eta(t_j)-\alpha_i}\right) \dot\eta(t_j)
+\alpha_i-a_i=0, \quad i=1,2,\ldots,\ell.
\end{equation}
\end{subequations}

Let $\bz$ be the vector of the $\ell n+\ell$ unknowns
$w_1,w_2,\ldots,w_{\ell n},a_1,\ldots,a_\ell$, i.e.,
\[
\bz=
\begin{bmatrix}
w_1 &
\ldots &
w_{\ell n} &
a_1 &
\ldots &
a_\ell
\end{bmatrix}^T \in\CC^{\ell n+\ell},
\]
and let the function $F$ be defined by
\[
F(\bz)=
\begin{bmatrix}
\sum_{j=1}^{\ell}m_j\log\left(\frac{w_1-a_j}{\eta(t_1)-\alpha_j}\right)-p_1 \\
\vdots \\
\sum_{j=1}^{\ell}m_j\log\left(\frac{w_{\ell n}-a_j}{\eta(t_{\ell 
n})-\alpha_j}\right)-p_{\ell n}	\\
\frac{1}{n\i}\sum_{j=1}^{\ell 
n}\log\left(\frac{w_j-a_1}{\eta(t_j)-\alpha_1}
\right)\dot\eta(t_j)+\alpha_1-a_1\\
\vdots	\\
\frac{1}{n\i}\sum_{j=1}^{\ell
n}\log\left(\frac{w_j-a_\ell}{\eta(t_j)-\alpha_\ell}\right)\dot\eta(t_j)
+\alpha_\ell-a_\ell\\
\end{bmatrix} \in \C^{\ell n + \ell}.
\]
Then the system of non-linear equations~\eqref{e:non} can be written as
\begin{equation}\label{eq:non-sys}
F(\bz)={\bf 0}.
\end{equation}

\subsubsection{Solving the non-linear system~(\ref{eq:non-sys})}

We shall solve the non-linear system~\eqref{eq:non-sys} using Newton's
iterative method
\begin{equation}\label{eq:non-new}
\bz^{k+1}=\bz^k-\left[F'(\bz^k)\right]^{-1}F(\bz^k), \quad k=0,1,2,\ldots\,,
\end{equation}
where $F'(\bz)$ is the Jacobian matrix of the function $F$ and is given by
\begin{equation*}
F'(\bz) = \begin{bmatrix} D & A_1 \\ A_2 & - I_\ell \end{bmatrix} \in \C^{\ell
n + \ell, \ell n + \ell},
\end{equation*}
where $I_\ell$ is the $\ell\times\ell$ identity matrix, and
\begin{align*}
D &=
\begin{bmatrix}
\sum_{j=1}^{\ell}\frac{m_j}{w_1-a_j} \\
  &\sum_{j=1}^{\ell}\frac{m_j}{w_2-a_j} \\
 &   &\ddots  \\
 &   &  &\sum_{j=1}^{\ell}\frac{m_j}{w_{\ell n}-a_j}
\end{bmatrix}
\in \C^{\ell n,\ell n}, 
\\
A_1 &=
\begin{bmatrix}
\frac{-m_1}{w_1-a_1} &\frac{-m_2}{w_1-a_2} &\ldots &\frac{-m_\ell}{w_1-a_\ell}\\
\frac{-m_1}{w_2-a_1} &\frac{-m_2}{w_2-a_2} &\ldots &\frac{-m_\ell}{w_2-a_\ell}\\
\vdots &\vdots &\ddots &\vdots\\
\frac{-m_1}{w_{\ell n}-\beta_1} &\frac{-m_2}{w_{\ell n}-a_2} &\ldots 
&\frac{-m_\ell}{w_{\ell n}-a_\ell}\\
\end{bmatrix}\in \C^{\ell n, \ell}, 
\\
A_2 &=
\begin{bmatrix}
\frac{1}{n\i}\frac{\dot\eta(t_1)}{w_1-a_1} 
&\frac{1}{n\i}\frac{\dot\eta(t_2)}{w_2-a_1} &\cdots 
&\frac{1}{n\i}\frac{\dot\eta(t_{\ell n})}{w_{\ell n}-a_1} \\
\frac{1}{n\i}\frac{\dot\eta(t_1)}{w_1-a_2} 
&\frac{1}{n\i}\frac{\dot\eta(t_2)}{w_2-a_2} &\cdots 
&\frac{1}{n\i}\frac{\dot\eta(t_{\ell n})}{w_{\ell n}-a_2} \\
\vdots & \vdots & \ddots & \vdots \\
\frac{1}{n\i}\frac{\dot\eta(t_1)}{w_1-a_\ell} 
&\frac{1}{n\i}\frac{\dot\eta(t_2)}{w_2-a_\ell} &\cdots 
&\frac{1}{n\i}\frac{\dot\eta(t_{\ell n})}{w_{\ell n}-a_\ell} 	\\
\end{bmatrix}
\in \C^{\ell,\ell n}. 
\end{align*}

Let us discuss the choice of the starting point $\bz^0$.
In our numerical examples, a good choice for the starting point of $a_j$ has
been found to be the center of mass of the boundary curve $\Gamma_j$, scaled by
some factor $> 1$.
A good choice for the starting point for the boundary values has been found to
be small circles around $a_j$.
In the following we assume that a suitable starting point $\bz^0$ for the
Newton method is used, so that, in particular, the matrix $F'(\bz^k)$ is
invertible in each iteration step.

For each iteration $k$ in~\eqref{eq:non-new}, it is required to solve the
linear system
\begin{equation}\label{eqn:original_sys}
F'(\bz^k) \bv = F(\bz^k)
\end{equation}
for $\bv \in \C^{\ell n + \ell}$.  By taking into account the block
structure of $F'(\bz)$, the system~\eqref{eqn:original_sys} can be reduced to
an $\ell \times \ell$ linear system, as we show next.

The vectors $\bv$ and $F(\bz^k)$ can be partitioned as
\begin{equation*}
\bv = \begin{bmatrix} \bx \\ \by \end{bmatrix}, \quad
F(\bz^k) = \begin{bmatrix} \bvb \\
\bc \end{bmatrix},
\end{equation*}
where $\bx,\bvb\in\C^{\ell n}$ and $\by,\bc\in\C^\ell$. Hence
equation~\eqref{eqn:original_sys} is equivalent to the linear system
\begin{equation}\label{eqn:equivalent_sys}
\begin{split}
D\bx + A_1\by &= \bvb, \\
A_2\bx - \by &= \bc.
\end{split}
\end{equation}

\begin{lemma}
The diagonal matrix $D = [d_{ij}] \in \C^{\ell n, \ell n}$ satisfies
\begin{equation} \label{eq:dii}
d_{ii} = \frac{U'(w_i)}{U(w_i)}, \quad i = 1, 2, \ldots, \ell n,
\end{equation}
where $U(w) = \prod_{j=1}^\ell (w-\beta_j)^{m_j}$, the $\beta_1, \beta_2,
\ldots, \beta_\ell$ are the last $\ell$ entries of $\bz^k$, and $m_1, \ldots,
m_\ell$ are the exponents of the lemniscatic domain $\cL$.
\end{lemma}

\begin{proof}
For the function $U$, we have
\[
\frac{U'(w)}{U(w)} = \frac{d}{dw} \log U(w)
= \frac{d}{dw} \sum_{j=1}^{\ell}m_j\log(w-a_j)
= \sum_{j=1}^{\ell}\frac{m_j}{w-a_j}.
\]
Substituting $w = w_i$ shows~\eqref{eq:dii}.
\eop
\end{proof}

In view of Lemma~\ref{lem:Upne0}, the lemma suggests that $D$ is non-singular
if $\bz^k$ is close to the solution of $F(\bz) = 0$.  This has also been
observed in all numerical experiments; see Section~\ref{sect:examples}.

If $D$ is not singular,
we can rewrite the first equation in~\eqref{eqn:equivalent_sys} as
\begin{equation}\label{eqn:x}
\bx = D^{-1} (\bvb - A_1\by )
\end{equation}
and insert this in the second equation in~\eqref{eqn:equivalent_sys} to obtain
\begin{equation}\label{eqn:y_sys}
( A_2 D^{-1} A_1 + I_\ell ) \by = A_2 D^{-1} \bvb - \bc.
\end{equation}
Now, we get a solution of~\eqref{eqn:original_sys} by solving the
$\ell \times \ell$ system~\eqref{eqn:y_sys} and computing $\bx$ 
by~\eqref{eqn:x},
instead of solving the $(\ell n + \ell) \times (\ell n + \ell)$
system~\eqref{eqn:original_sys} directly. The $\ell \times \ell$
system~\eqref{eqn:y_sys} can be solved using a direct method such as the
Gauss elimination method since $\ell$ is usually small.
If $D$ is non-singular, the matrix $A_2 D^{-1} A_1 + I_\ell$ is invertible if
and only if $F'(\bz^k)$ is invertible.  Thus the system~\eqref{eqn:y_sys}
is then uniquely solvable.

Next we show that it is possible to compute the vectors $\bx$ in~\eqref{eqn:x}
and $\by$ in~\eqref{eqn:y_sys} \emph{without} forming the matrices $A_1$, $A_2$
or $D$ first, by taking into account the Cauchy structure of $A_1$ and $A_2$.
Indeed, with the Cauchy matrix
\begin{equation*}
C
= \begin{bmatrix}
\frac{1}{w_1-a_1} & \frac{1}{w_1-a_2} & \ldots & \frac{1}{w_1-a_\ell} \\
\frac{1}{w_2-a_1} & \frac{1}{w_2-a_2} & \ldots & \frac{1}{w_2-a_\ell} \\
\vdots & \vdots & \ddots & \vdots \\
\frac{1}{w_{\ell n}-a_1} & \frac{1}{w_{\ell n}-a_2} & \ldots &
\frac{1}{w_{\ell n}-a_\ell}
\end{bmatrix}
= \begin{bmatrix} \frac{1}{w_r - a_s} \end{bmatrix}_{r, s} \in \C^{\ell n,
\ell},
\end{equation*}
the matrices $A_1$ and $A_2$ can be written as
\begin{equation*}
A_1 = - C \begin{bmatrix} m_1 \\ & m_2 \\ & & \ddots \\ & & & m_\ell
\end{bmatrix},
A_2 = \frac{1}{n \i} C^T \begin{bmatrix} \dot\eta(t_1) \\ & \dot\eta(t_2) \\ &
&
\ddots \\ & & & \dot\eta(t_{\ell n}) \end{bmatrix}.
\end{equation*}
Then the matrix $A_2 D^{-1} A_1$ in the linear system~\eqref{eqn:y_sys} can
be written as
\begin{equation*}
A_2 D^{-1} A_1 = \frac{\i}{n} \begin{bmatrix} \sum_{k=1}^{\ell n} 
\frac{m_s}{(w_k-a_r)
(w_k-a_s)} \frac{\dot{\eta}(t_k)}{ \sum_{j=1}^\ell \frac{m_j}{w_k-a_j} }
\end{bmatrix}_{r, s} \in \C^{\ell, \ell},
\end{equation*}
so that we can generate the entries of this matrix directly, without first
forming $A_1$ or $A_2$.  Similarly we can write the right-hand side
of~\eqref{eqn:y_sys} as
\begin{equation*}
A_2 D^{-1}\bvb - \bc = \frac{1}{n \i} \begin{bmatrix} \sum_{k=1}^{\ell n}
\frac{b_{k} }{w_k - a_r} \frac{\dot{\eta}(t_k)}{ \sum_{j=1}^\ell
\frac{m_j}{w_k-a_j} }
\end{bmatrix}_{r = 1, 2, \ldots, \ell} - \bc \in \C^\ell,
\end{equation*}
with $\bvb = [ b_{1}, b_{2}, \ldots, b_{\ell n}]^T$. Finally, for computing the
vector $\bx$, we have from~\eqref{eqn:x},
\begin{equation*}
\bx=D^{-1} ( \bvb - A_1\by) =\begin{bmatrix} \frac{1 }{ \sum_{j=1}^\ell
\frac{m_j}{w_r-a_j} }
\left(b_r+\sum_{k=1}^{\ell} \frac{m_ky_k}{w_r-a_k}\right)
\end{bmatrix}_{r= 1, 2, \ldots, \ell n}
\end{equation*}
with $\by = [ y_{1}, y_{2}, \ldots, y_{\ell}]^T$.

The pseudo-code in Algorithm~\ref{alg:method} summarizes our method described 
in Sections~\ref{sect:compute_mv_tau}--\ref{sect:compute_aj_Phi}.

\begin{algorithm}[h!]
\caption{Pseudo-code of the method}\label{alg:method}
\textbf{Input:} discretization $\bt$ of $J$ as in~\eqref{e:ti-skp}, 
parametrization $\eta(\bt), \dot{\eta}(\bt)$ of the boundary of $\cK$, and 
auxiliary points $\alpha_j$ in the interior of the boundary curves $\Gamma_j$ 
($j = 1, \ldots, \ell$). \\
\textbf{Output:} boundary values $\Phi(\eta(\bt))$ and parameters $a_1, \ldots, 
a_\ell, m_1, \ldots, m_\ell, \tau$ of the lemniscatic domain.

\bigskip

\begin{compactenum}
\item[] For $j = 1, \ldots, \ell$:
\begin{compactenum}
\item[] Compute $\gamma_j(\bt) = -\log \abs{\eta(\bt)-\alpha_j}$ 
from~\eqref{eq:gam-j}.
\item[] Compute $\mu_j(\bt)$ and $h_j(\bt)$ from the boundary integral 
equation~\eqref{eq:ie-j} and \eqref{eq:h-j} with the MATLAB function 
\verb|fbie|.
\end{compactenum}
\item[] End
\item[] Compute $m_1, \ldots, m_\ell, \log\tau$ by solving the linear algebraic 
system~\eqref{eq:pj-sys}.
\item[] Compute the $p_i$ from~\eqref{eqn:bdry_vals}, where $\gamma(\bt)$ and 
$\mu(\bt)$ are given by~\eqref{eq:gam} and~\eqref{eq:mu-mu-j}.
\item[] Compute the boundary values of $\Phi$ and $a_1,\dots,a_\ell$ by solving 
the nonlinear system~\eqref{eq:non-sys} with Newton's method.
\end{compactenum}
\end{algorithm}

We have used this method in the numerical experiments shown in 
Section~\ref{sect:examples}.

\subsection{Computation of the interior values of $\Phi$}

The method described above yields boundary values of the function 
$\Phi$, namely the values $\Phi(\eta(t))$ at the points $t_i$ for 
$i=1,2,\dots,\ell n$. The values of $\Phi$ at interior points $z\in \KK$ can be 
computed by Cauchy's integral formula applied to the function $\Phi(z) - z$, 
which is analytic throughout $\cK$ and vanishes at $\infty$,
\begin{equation*}
\Phi(z) = z + \frac{1}{2\pi\i}\int_J
\frac{ (\Phi(\eta(t)) - \eta(t)) \, \dot\eta(t)}{\eta(t)-z} \, dt.
\end{equation*}
A fast and accurate method to compute the Cauchy integral formula has been
given in~\cite{Nas-fast} (see also~\cite{Aus13,Hel-Oja,Nas-siam13}). The method
is based on using the MATLAB function \verb|zfmm2dpart| in~\cite{Gre-Gim12}. To
compute the Cauchy integral formula at $p$ interior points, the method requires
$O(p+\ell n)$ operations.

\begin{remark}
The Cauchy integral formula can also be used to compute the values of the 
inverse mapping $\Phi^{-1}$ for interior points $w\in\LL$~\cite[p.~380]{Hen3}. 
However, it requires the boundary values of both the function $\Phi$
and its derivative $\Phi'$, i.e., $\Phi(\eta(t))$ and $\Phi'(\eta(t))$ at 
the points $t_i$ for $i=1,2,\dots,\ell n$. For computing the boundary values 
of $\Phi'$, we can use the boundary integral equation with the adjoint 
Neumann kernel as in~\cite{Yun-proc14}. Alternatively, we can compute 
$\Phi'(\eta(t))$ numerically from $\Phi(\eta(t))$.
\end{remark}

\section{Numerical examples}
\label{sect:examples}

In this section we present numerical examples for five domains that illustrate 
our method.

\begin{example}\label{ex:1}
{\rm We consider the unbounded domain $\cK$ exterior to the two circles
\begin{equation*}
\eta_{1,2}(t)= \pm 1 + r e^{-\i t}, \quad 0 \leq t \leq 2\pi, \quad 0 < r < 1,
\end{equation*}
for different values of the radius $r$.  The corresponding conformal
map $\Phi$ and lemniscatic domain
\begin{equation*}
\cL = \{ w \in \widehat{\C} : |w-a_1|^{m_1}\,|w-a_2|^{m_2}>\tau \}
\end{equation*}
(both depending on the value of $r$) have been derived analytically
in~\cite[Section~4]{Set-Lie15}.  We therefore can compare our numerically 
computed parameters with the exact values $a_1$, $a_2$, $m_1$, $m_2$, $\tau$ 
defining $\cL$. 
Figure~\ref{f:ex-1} shows the domains $\cK$ for $r=0.5$, $r=0.7$, 
and $r=0.9$ and the corresponding lemniscatic domains $\cL$.

We denote the numerically computed parameters of the lemniscatic domain by
$a_{1,n}$, $a_{2,n}$, $m_{1,n}$ ,$m_{2,n}$, $\tau_n$, where $n$ is the number 
of nodes
in the discretization of each boundary component. The (absolute) errors
\begin{align*}
E_{a,n} &= \max\{|a_1-a_{1,n}|,|a_2-a_{2,n}|\}, \\
E_{m,n} &= \max\{|m_1-m_{1,n}|,|m_2-m_{2,n}|\}, \\
E_{\tau,n} &= |\tau-\tau_n|
\end{align*}
are shown in Figure~\ref{f:ex-1-ex-err}.  We observe that all errors are
quite small already for a small number of nodes. In fact, with $n = 2^6 = 64$ 
nodes the errors in this example are close to the machine precision level of 
$10^{-16}$.  Increasing the number of nodes beyond this point leads to some 
irregularities in the observed convergence behavior, which most likely is due 
to the some slight differences in the accuracy of the computed solution of the 
respective linear algebraic systems.  Since all errors remain on the order of 
$10^{-14}$ or smaller, we did not further investigate this phenomenon. 

Our method requires solving one linear algebraic system with the matrix $I-B$
of size $\ell n \times \ell n$ with a different right hand side for each 
boundary component.  In this example we have $\ell=2$, and hence there are two 
linear algebraic systems to be solved for each value of $r$ and $n$.  As 
described in Section~\ref{sect:compute_mv_tau}, we use the (full and 
unpreconditioned) GMRES method for this task. In Figure~\ref{fig:gmres_rr} we 
plot \emph{all} relative residual norms of the GMRES method we obtained for the 
two linear algebraic systems, $r = 0.5, 0.7, 0.9$, and $n=2^6, 2^7, \dots, 
2^{10}$.  Thus, Figure~\ref{fig:gmres_rr} shows the GMRES convergence for $30$ 
linear algebraic systems.  We observe that the number of GMRES iteration steps 
required to attain a relative residual norm on the order of $10^{-14}$ is 
\emph{very} small and almost independent of parameters in the linear algebraic 
systems (namely the right hand side, $r$, and $n$). We have indicated 
reasons for this observation in Section~\ref{sect:compute_mv_tau}, but, as 
mentioned there, a detailed analysis is the subject of future work.

The 2-norm condition numbers of the matrices $D$ and $A_2D^{-1}A_1+I_{\ell}$ 
in the Newton iteration are shown in Figure~\ref{f:ex-1-cond}.  
All matrices stay quite well conditioned throughout the iteration.
The same observation can be made in the following examples as well. 
Figure~\ref{f:ex-1-err} shows the norms $\|\bz^{k+1}-\bz^{k}\|_\infty$ 
in the Newton iteration.  
}\end{example}

\begin{figure}
\centerline{
\includegraphics[width=0.5\textwidth]{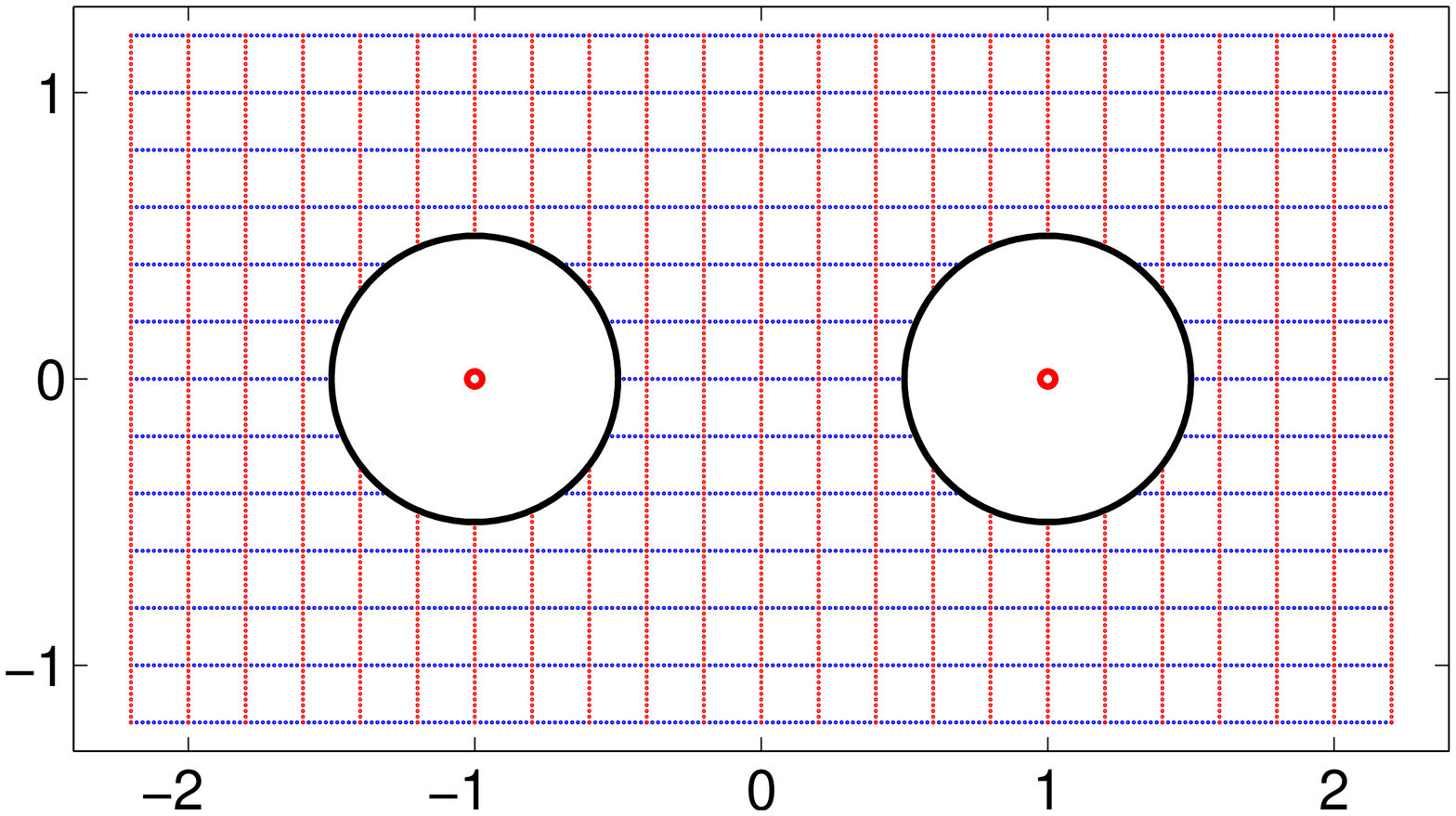}
\includegraphics[width=0.5\textwidth]{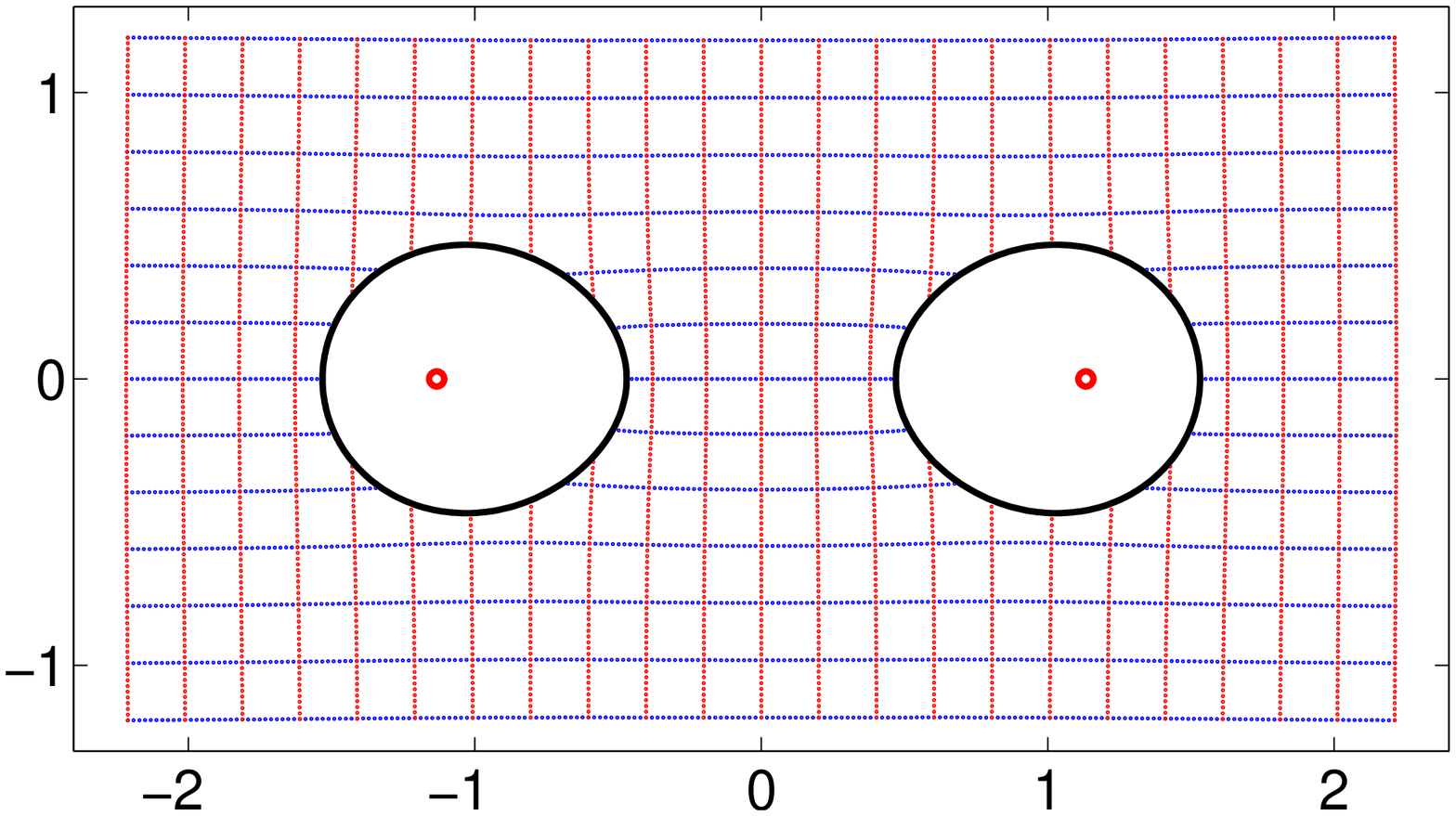}
}
\centerline{
\includegraphics[width=0.5\textwidth]{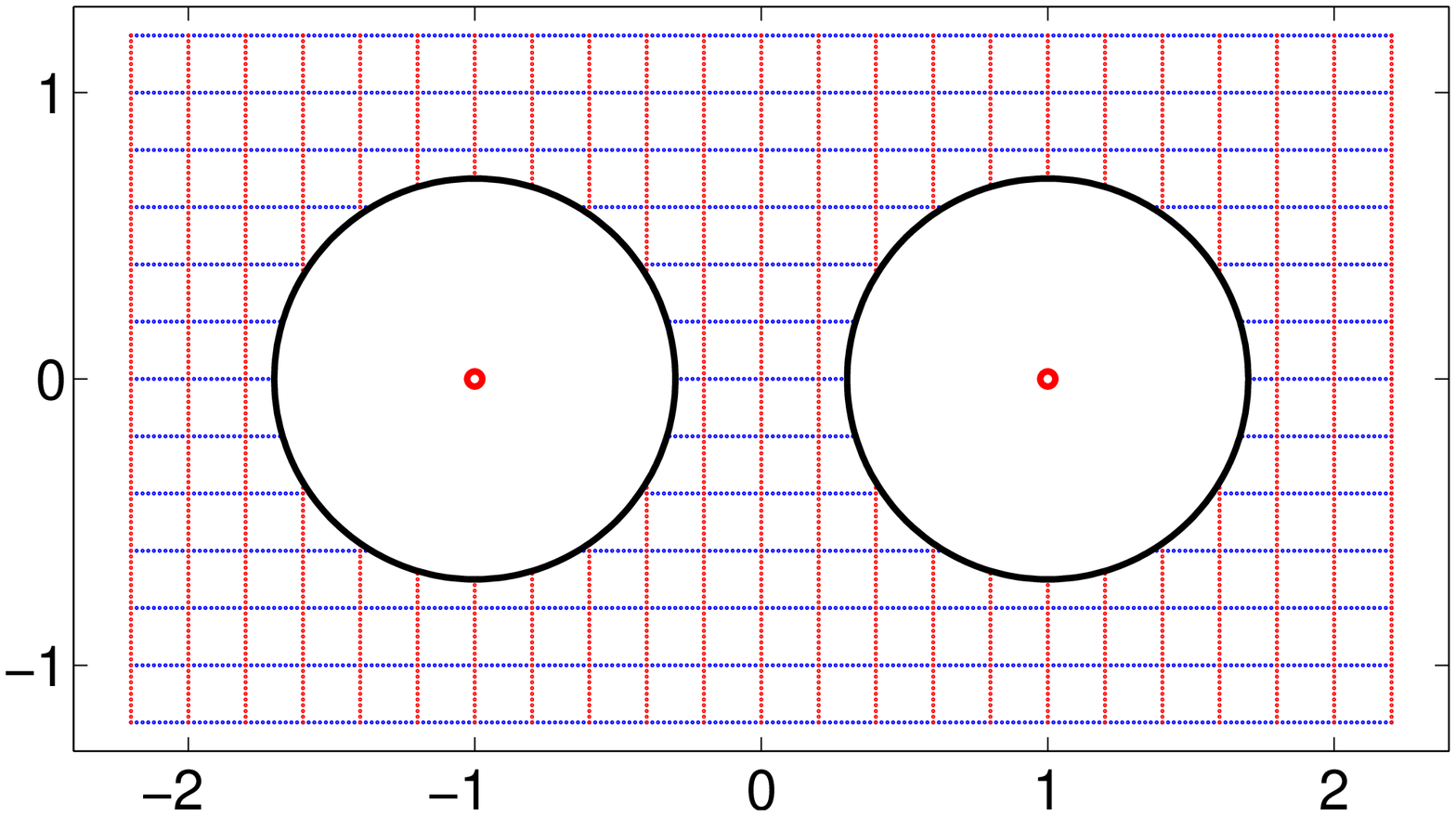}
\includegraphics[width=0.5\textwidth]{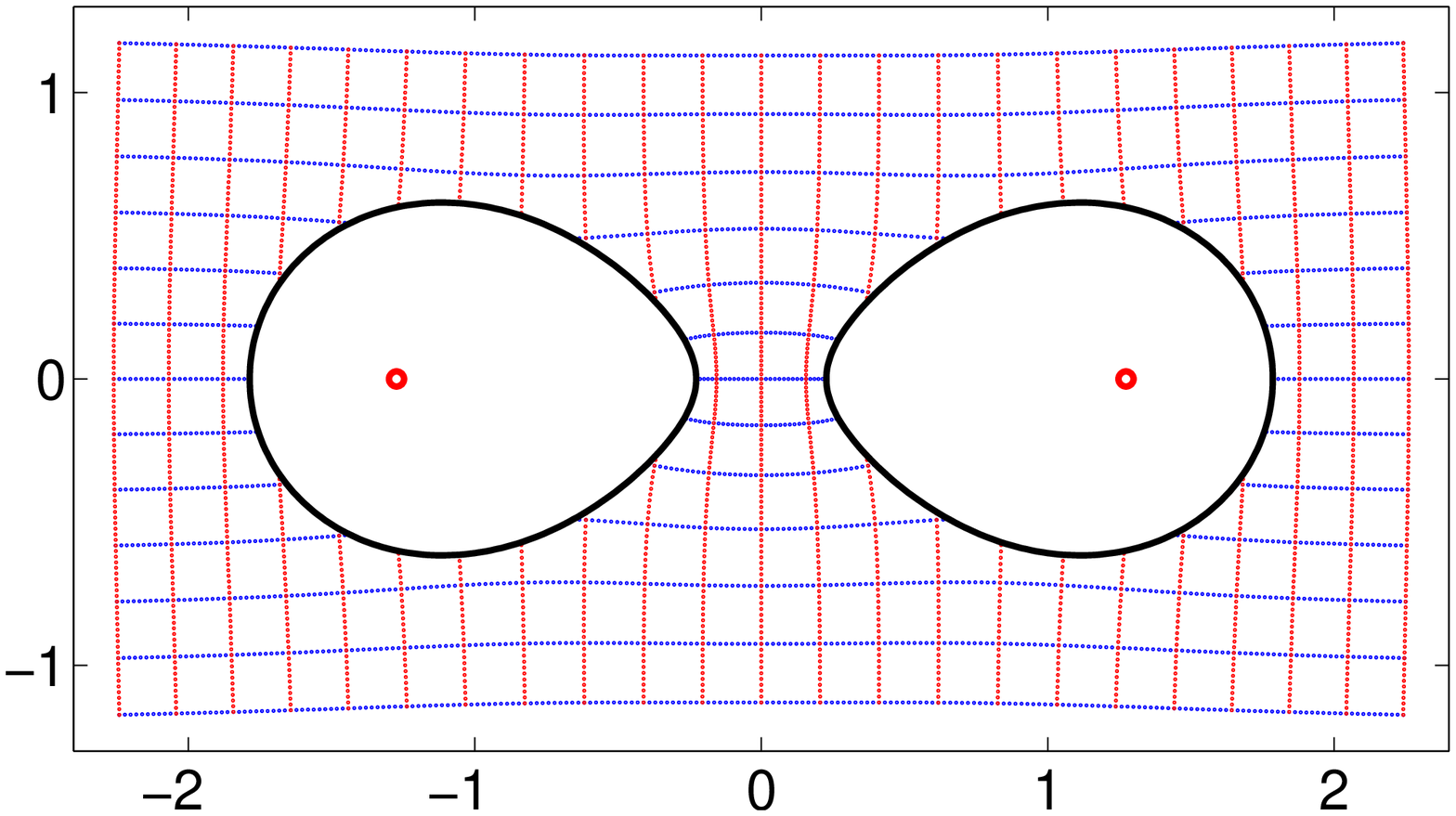}
}
\centerline{
\includegraphics[width=0.5\textwidth]{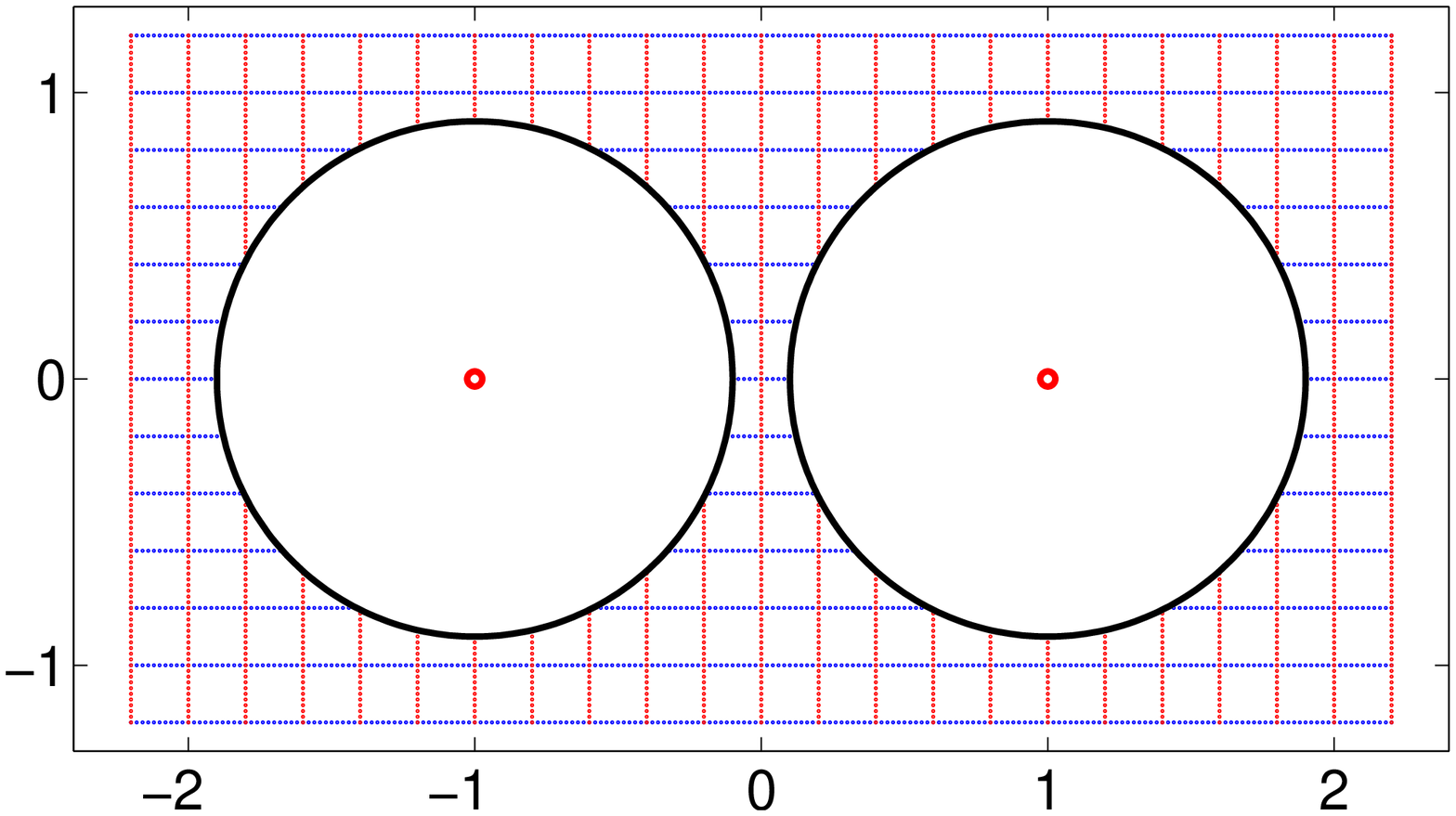}
\includegraphics[width=0.5\textwidth]{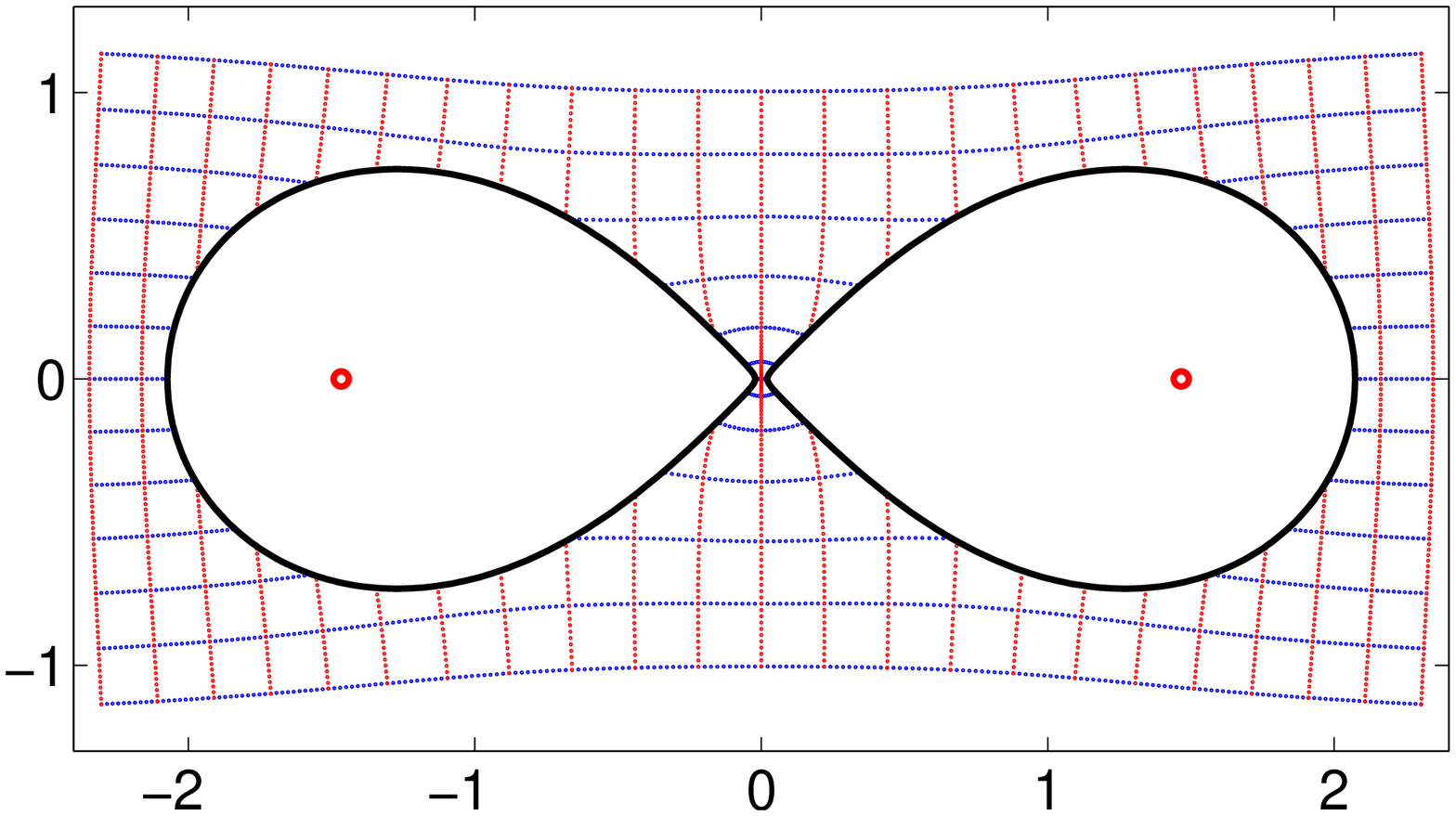}
}
\caption{Original domains for Example~\ref{ex:1} (left) and 
corresponding lemniscatic domains (right) obtained with $n=256$ 
and for $r=0.5$ (top), $r=0.7$ (middle), and $r=0.9$ (bottom).}
\label{f:ex-1}
\end{figure}

\begin{figure}
\centerline{
\includegraphics[width=0.5\textwidth]{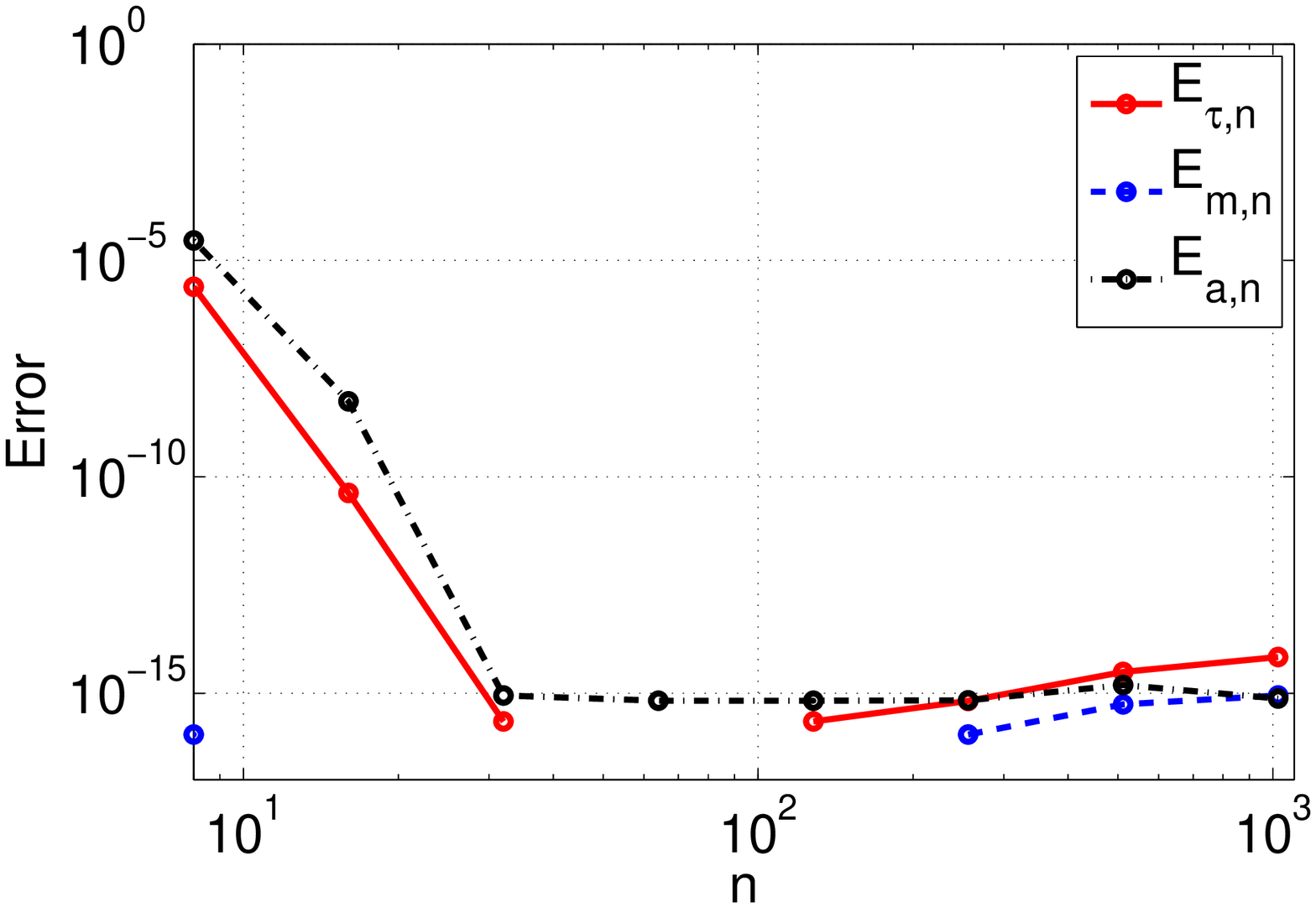}
\includegraphics[width=0.5\textwidth]{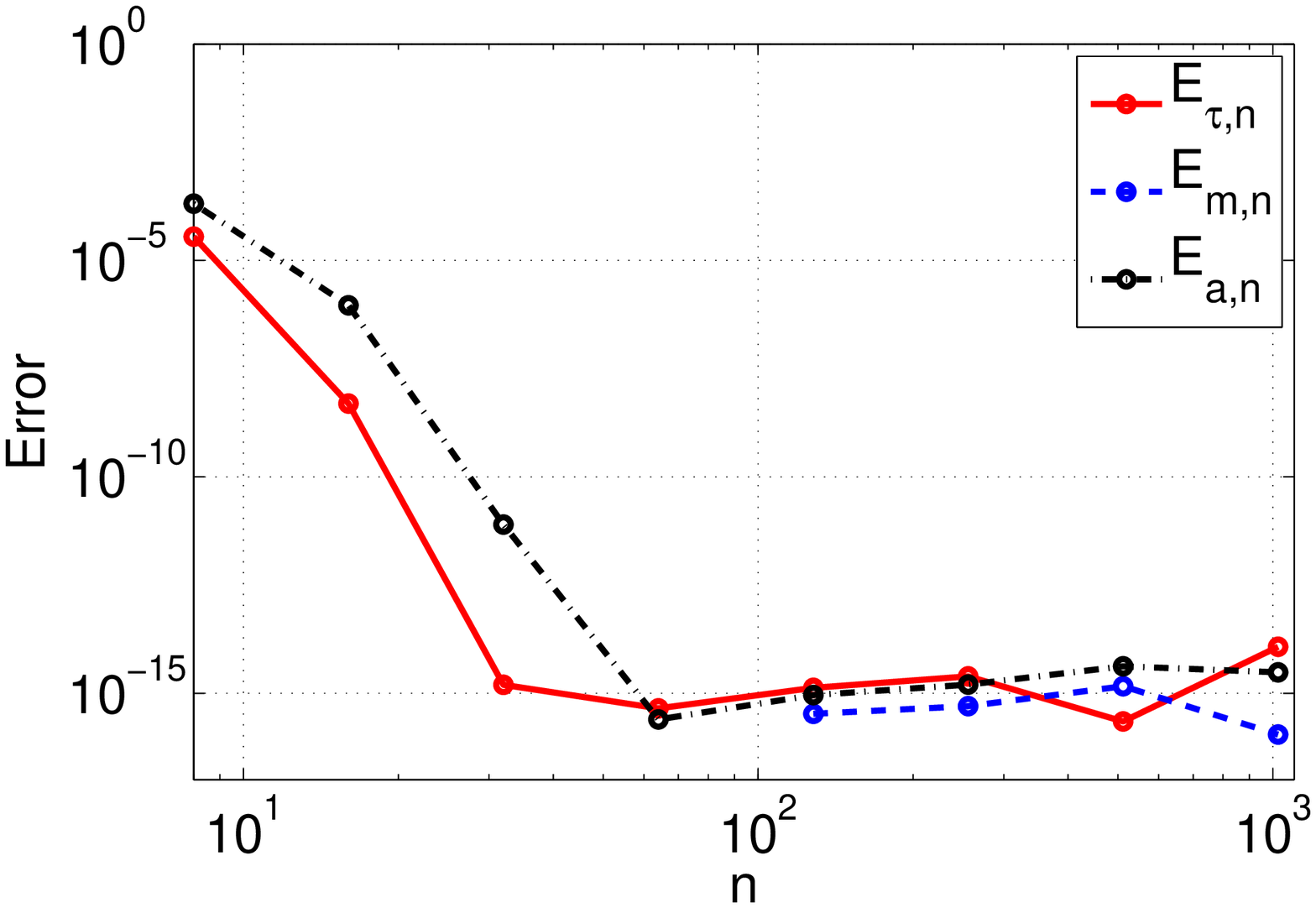}
}
\centerline{
\includegraphics[width=0.5\textwidth]{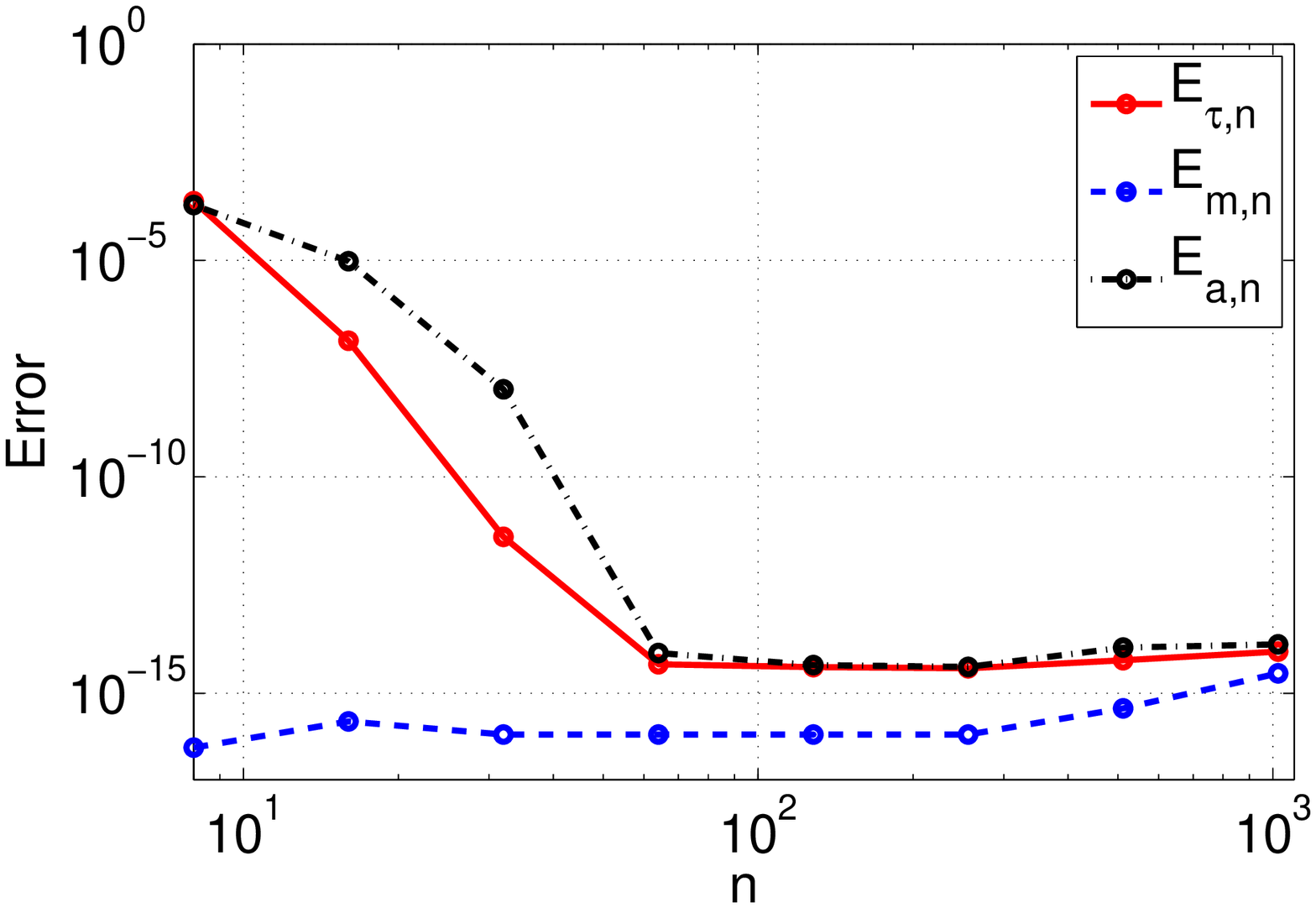}
}
\caption{Errors $E_{a,n}$, $E_{m,n}$, and $E_{\tau,n}$ obtained with $n$ nodes
and for $r=0.5$ (top left), $r=0.7$ (top right), and $r=0.9$ (bottom) 
in Example~\ref{ex:1}.  Missing dots indicate that the error is zero.}
\label{f:ex-1-ex-err}
\end{figure}

\begin{figure} 
\centerline{
\includegraphics[width=0.5\textwidth]{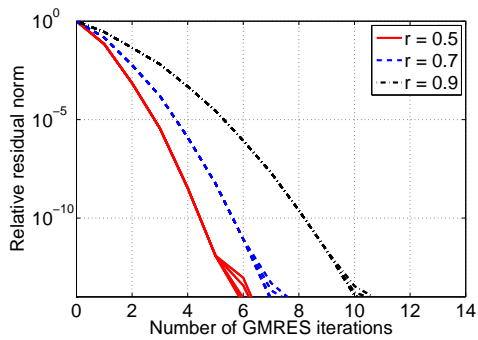}
}
\caption{Relative residual norms of the GMRES method for different values of 
$r$ and $n$ in Example~\ref{ex:1}.}
\label{fig:gmres_rr}
\end{figure}

\begin{figure}
\centerline{
\includegraphics[width=0.5\textwidth]{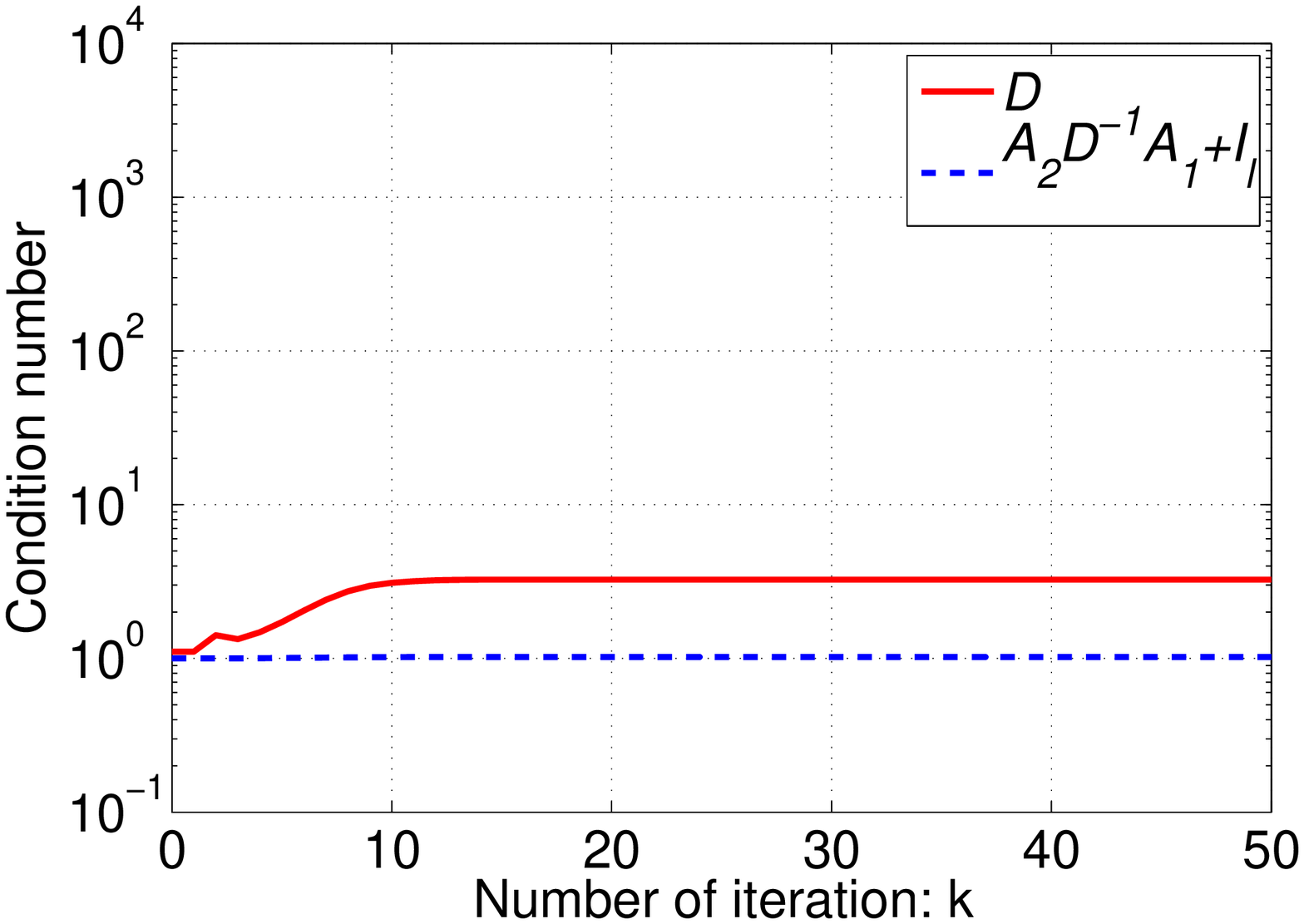}
\includegraphics[width=0.5\textwidth]{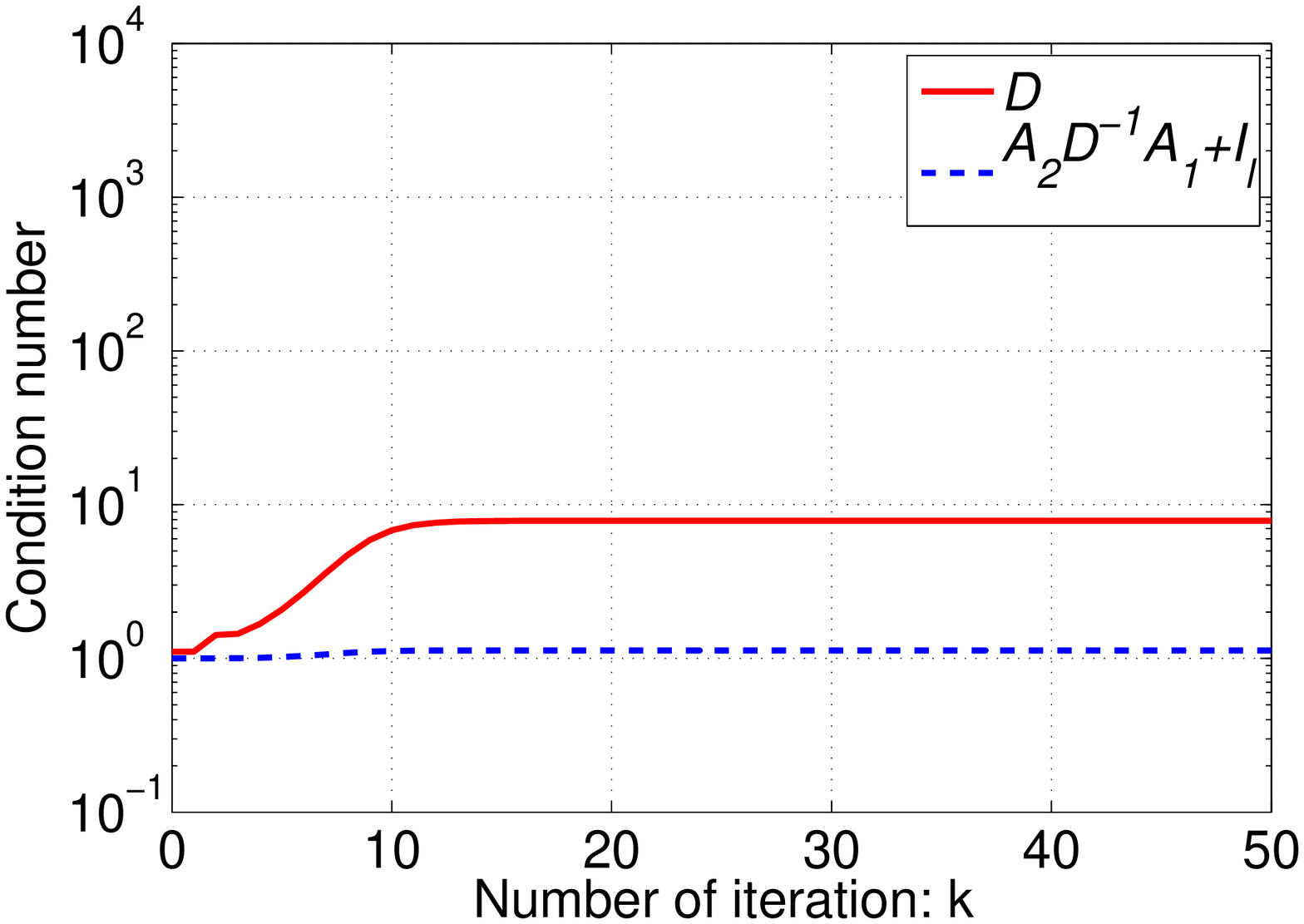} }
\centerline{
\includegraphics[width=0.5\textwidth]{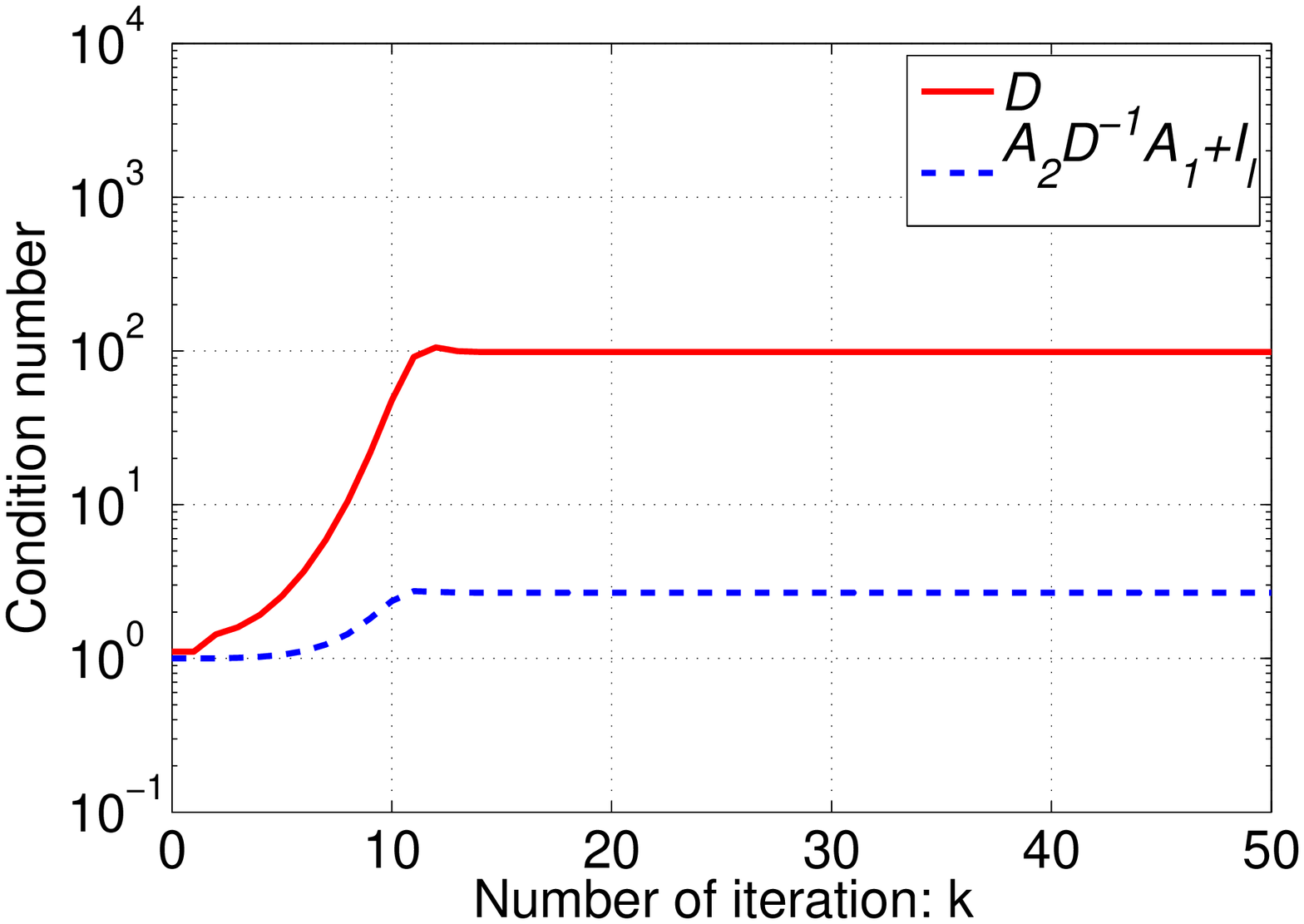}
}
\caption{2-norm condition numbers of $D$ and $A_2D^{-1}A_1+I_{\ell}$ 
obtained with $n=256$ and for $r=0.5$ (top left), $r=0.7$ (top right),
and $r=0.9$ (bottom) in Example~\ref{ex:1}.}
\label{f:ex-1-cond}
\end{figure}

\begin{figure}
\centerline{
\includegraphics[width=0.5\textwidth]{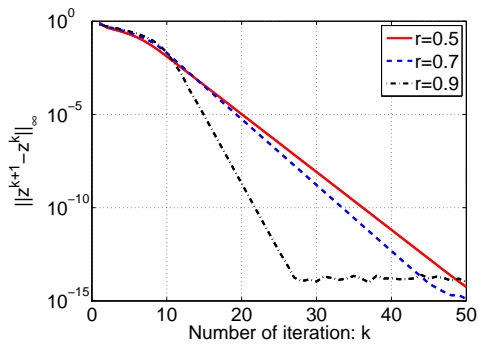}
}
\caption{$\|\bz^{k+1}-\bz^{k}\|_\infty$ in the Newton iteration obtained 
with $n=256$ and for $r=0.5$, $r=0.7$, and $r=0.9$ in Example~\ref{ex:1}.}
\label{f:ex-1-err}
\end{figure}

\begin{example}\label{ex:2}
{\rm We consider the unbounded domain $\cK$ exterior to seven nonconvex and 
complicated but smooth curves as shown in Figure~\ref{f:ex-2}.  These curves 
are parametrized (from left to right) by
\[
\eta_j(t)=r_j(t)\, e^{-\i t}, \quad 0\le t\le 2\pi, \quad j=1,2,\ldots,7,
\]
where
\begin{align*}
r_1(t) &= 1.25+0.50\sin(4t)+0.30\cos(t), \\
r_2(t) &= 1.25+0.40\sin(2t)+0.20\cos(3t),\\
r_3(t) &= 0.75+0.25e^{\cos(t)}\cos^2(3t)+0.50e^{\sin(t)}\sin^2(2t),\\
r_4(t) &= e^{\cos(t)}\cos^2(2t)+e^{\sin(t)}\sin^2(2t),\\
r_5(t) &= 0.75+0.25e^{\cos(t)}\cos^2(2t)+0.50e^{\sin(t)}\sin^2(3t),\\
r_6(t) &= 1.25+0.40\sin(4t)+0.20\cos(3t), \\
r_7(t) &= 1.25+0.50\sin(3t)+0.30\cos(t).
\end{align*}
The computed lemniscatic domain obtained with $n=256$ 
is shown on the bottom of Figure~\ref{f:ex-2}.
The GMRES method for the seven linear algebraic systems required $36$ iteration 
steps to attain a residual norm smaller than $10^{-14}$.
Figure~\ref{f:ex-2-err} shows the 2-norm condition numbers of the matrices $D$ 
and $A_2D^{-1}A_1+I_{\ell}$ as well as the norms $\|\bz^{k+1}-\bz^{k}\|_\infty$
in the Newton iteration.}
\end{example}

\begin{figure}
\centerline{\scalebox{0.4}{\includegraphics{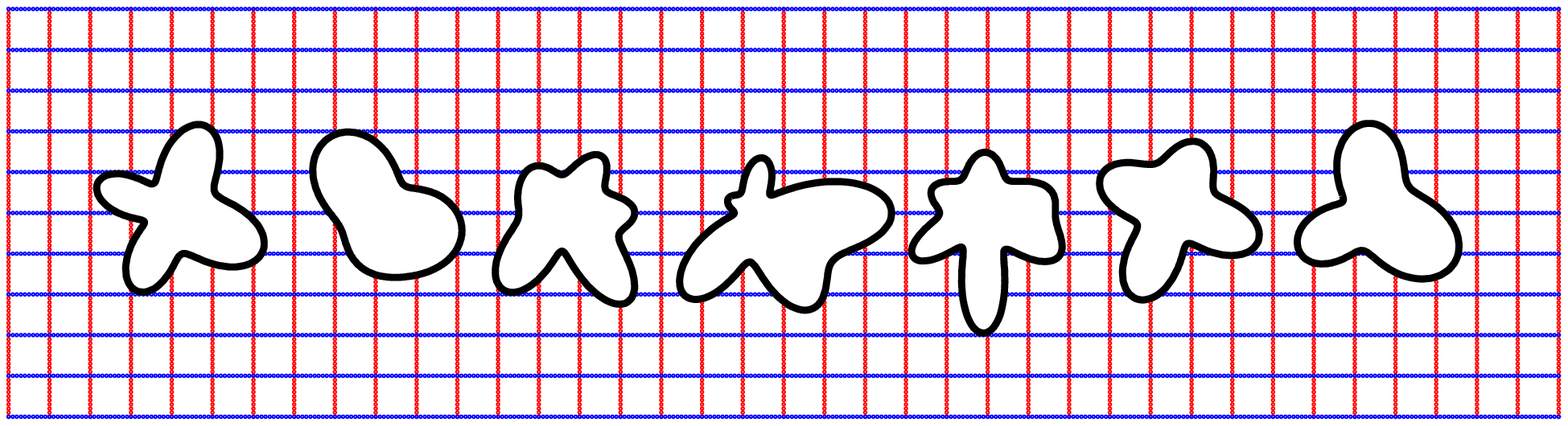}}}
\centerline{\scalebox{0.4}{\includegraphics{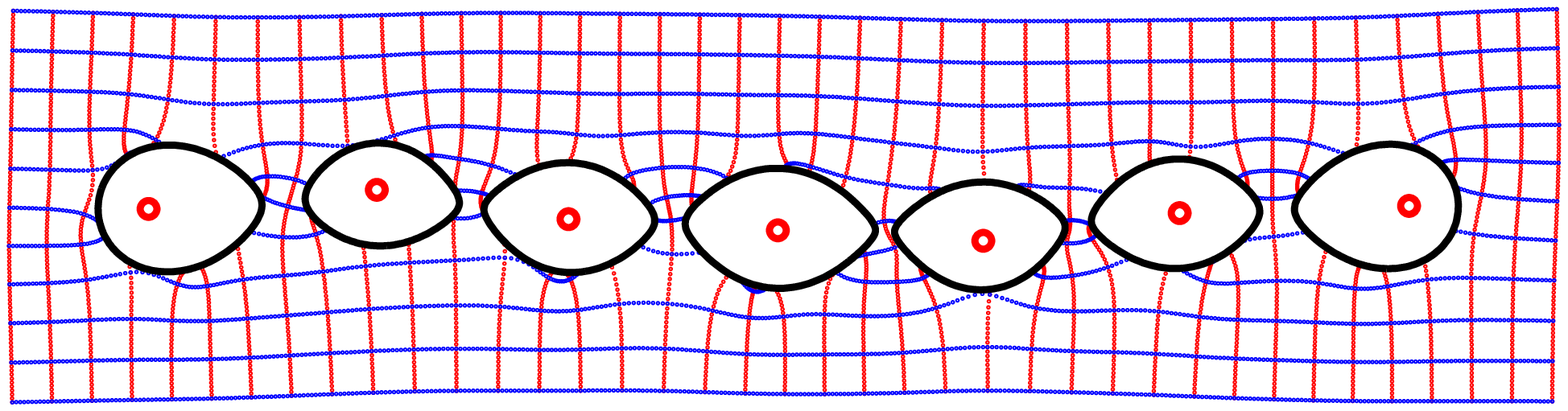}}}
\caption{Original domain for Example~\ref{ex:2} (top) and corresponding 
lemniscatic domain (bottom) obtained with $n=256$.}
\label{f:ex-2}
\end{figure}

\begin{figure}
\centerline{
\includegraphics[width=0.5\textwidth]{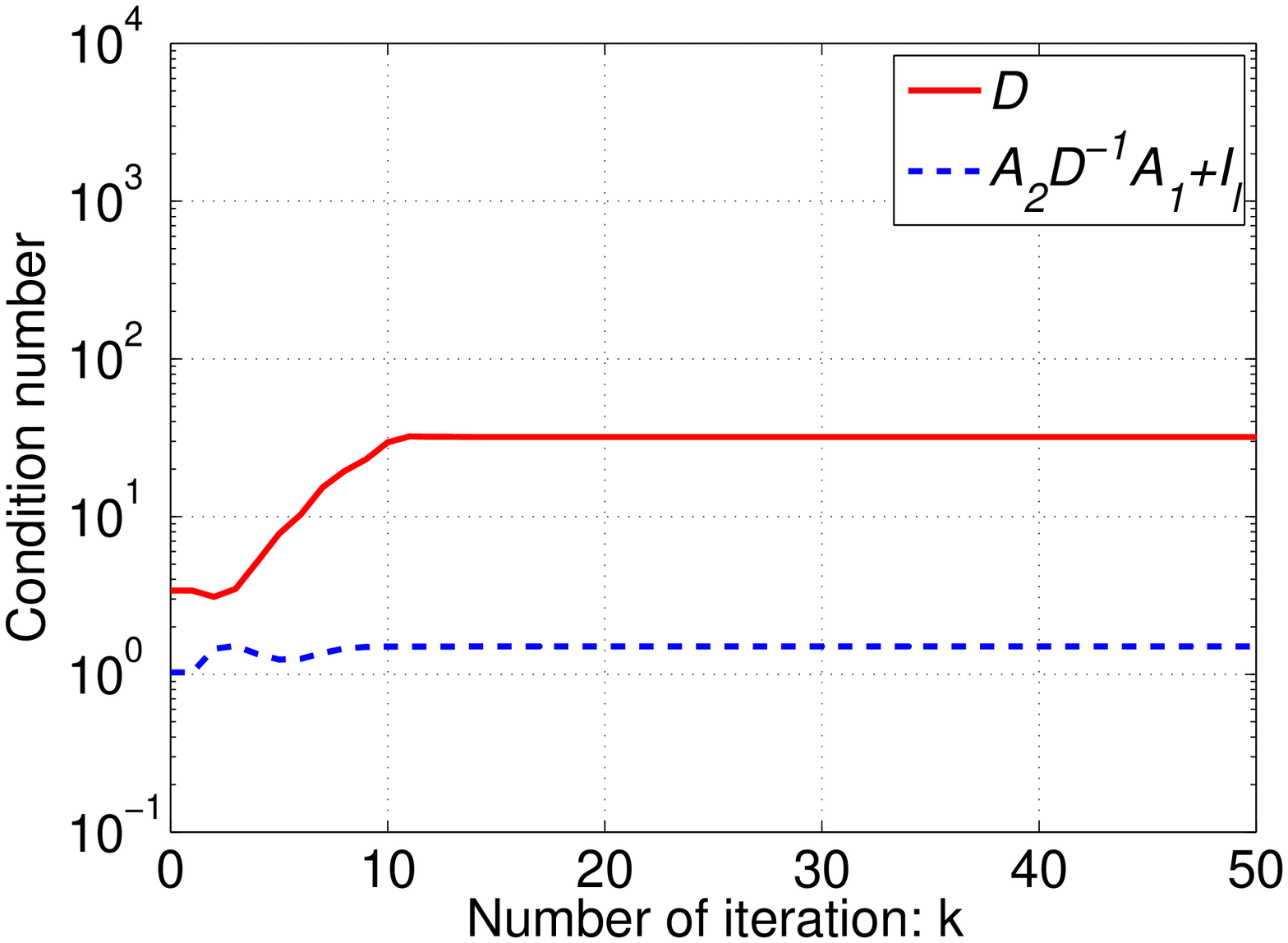}
\includegraphics[width=0.5\textwidth]{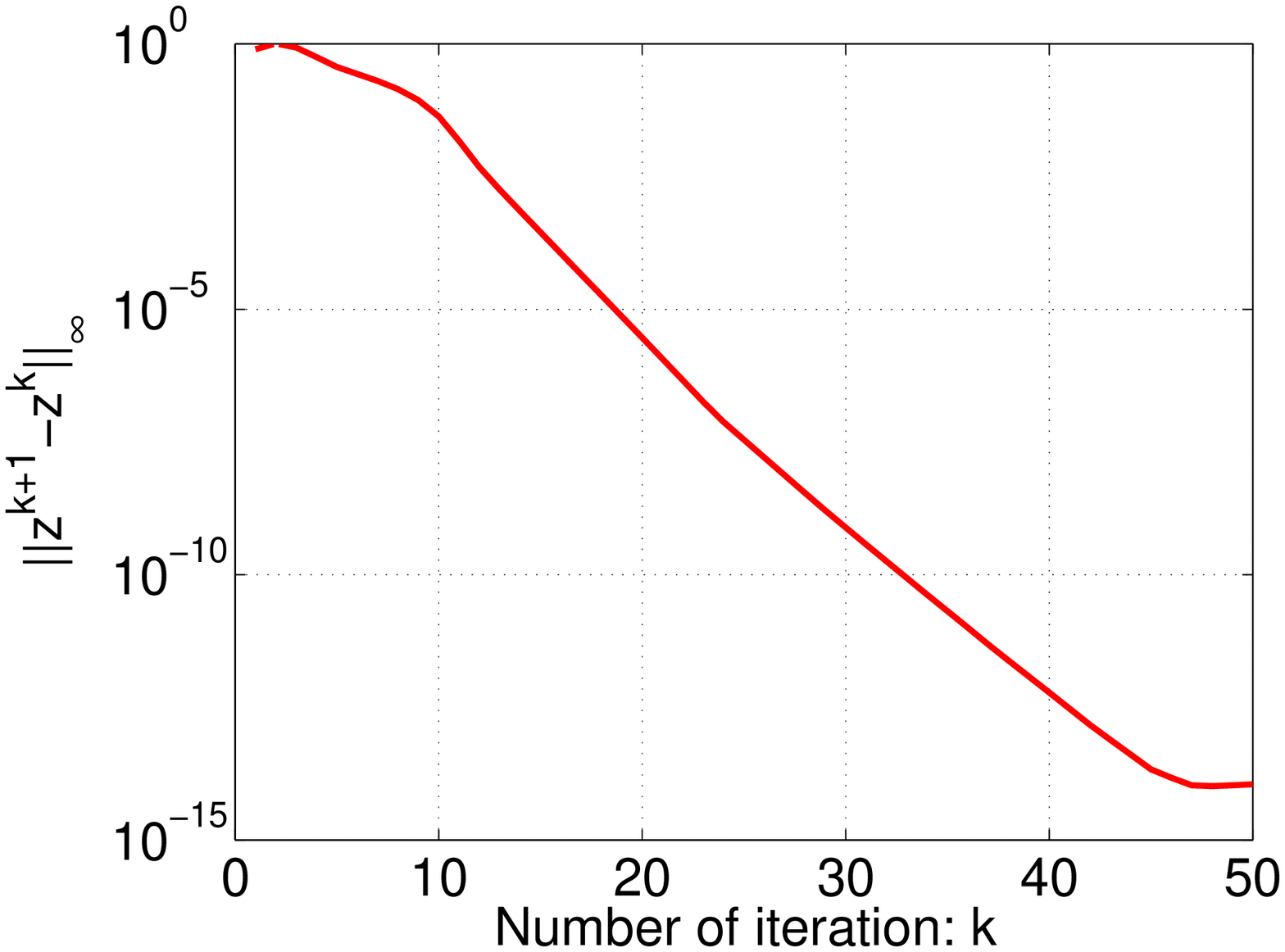}
}
\caption{2-norm condition numbers of $D$ and $A_2D^{-1}A_1+I_{\ell}$ (left)
and norms $\|\bz^{k+1}-\bz^{k}\|_\infty$ (right) in Example~\ref{ex:2}.}
\label{f:ex-2-err}
\end{figure}


\begin{example}\label{ex:3}
{\rm In this example we demonstrate that our method also works for domains with 
high connectivity. We consider the unbounded domain $\cK$ exterior to 64
circles as shown on the left of Figure~\ref{f:ex-3}. 
The domain $\cK$ is symmetric with respect to both the
real and the imaginary axis. Since the computed conformal map is normalized
as in~\eqref{eq:norm} the lemniscatic domain $\cL$ has the same symmetry 
properties; see~\cite[Lemma~2.2]{Set-Lie15}.
The computed lemniscatic
domain obtained with $n=256$ is shown on the right of that figure.
The GMRES method for the 64 linear algebraic systems required between 14 and 18 
iteration steps to attain a residual norm smaller than $2 \cdot 10^{-14}$.
Figure~\ref{f:ex-3-err} shows the 2-norm condition numbers of $D$ and
$A_2D^{-1}A_1+I_{\ell}$ as well as the norms $\|\bz^{k+1}-\bz^{k}\|_\infty$ 
in the Newton iteration.
}\end{example}

\begin{figure}
\centerline{
\includegraphics[width=0.45\textwidth]{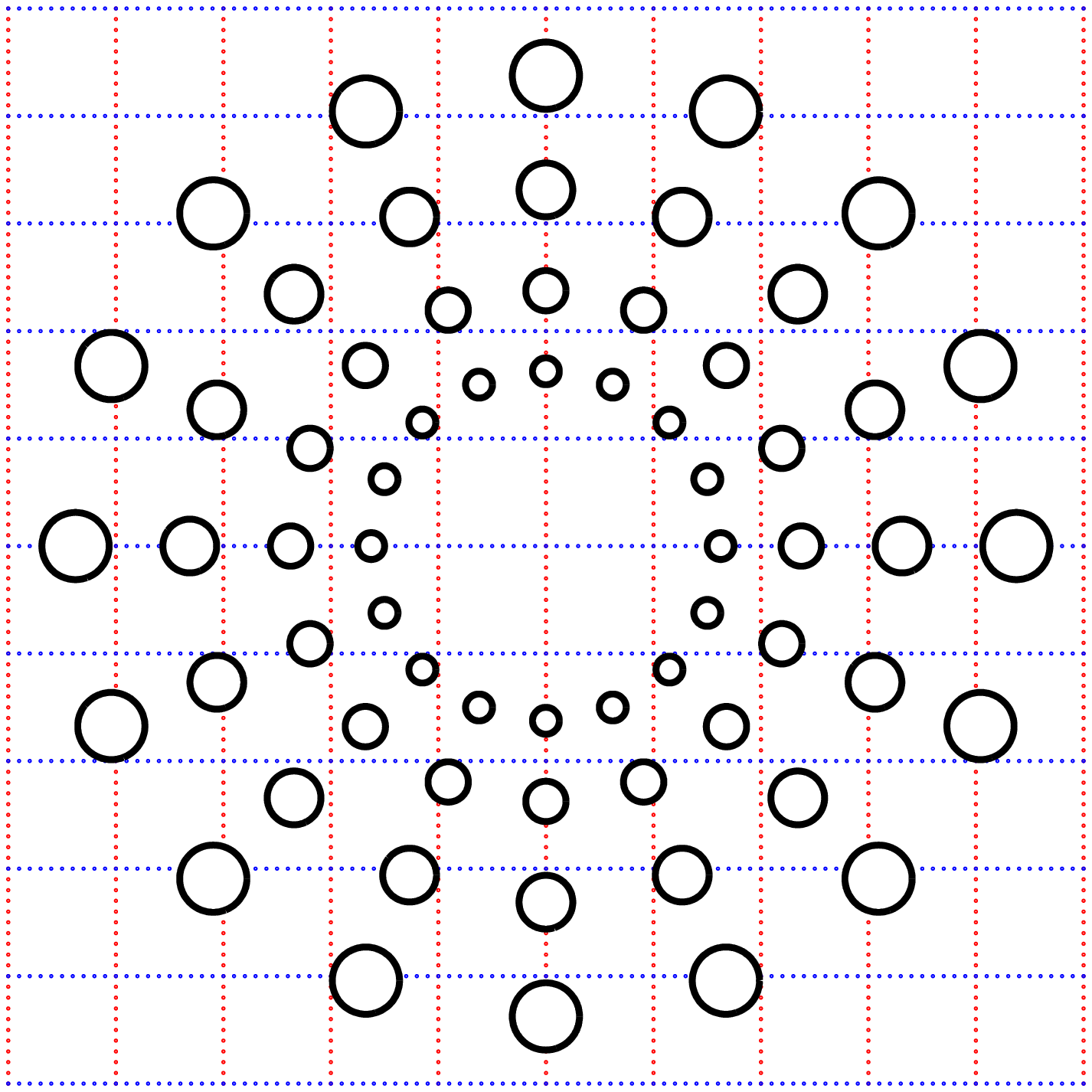}
\includegraphics[width=0.45\textwidth]{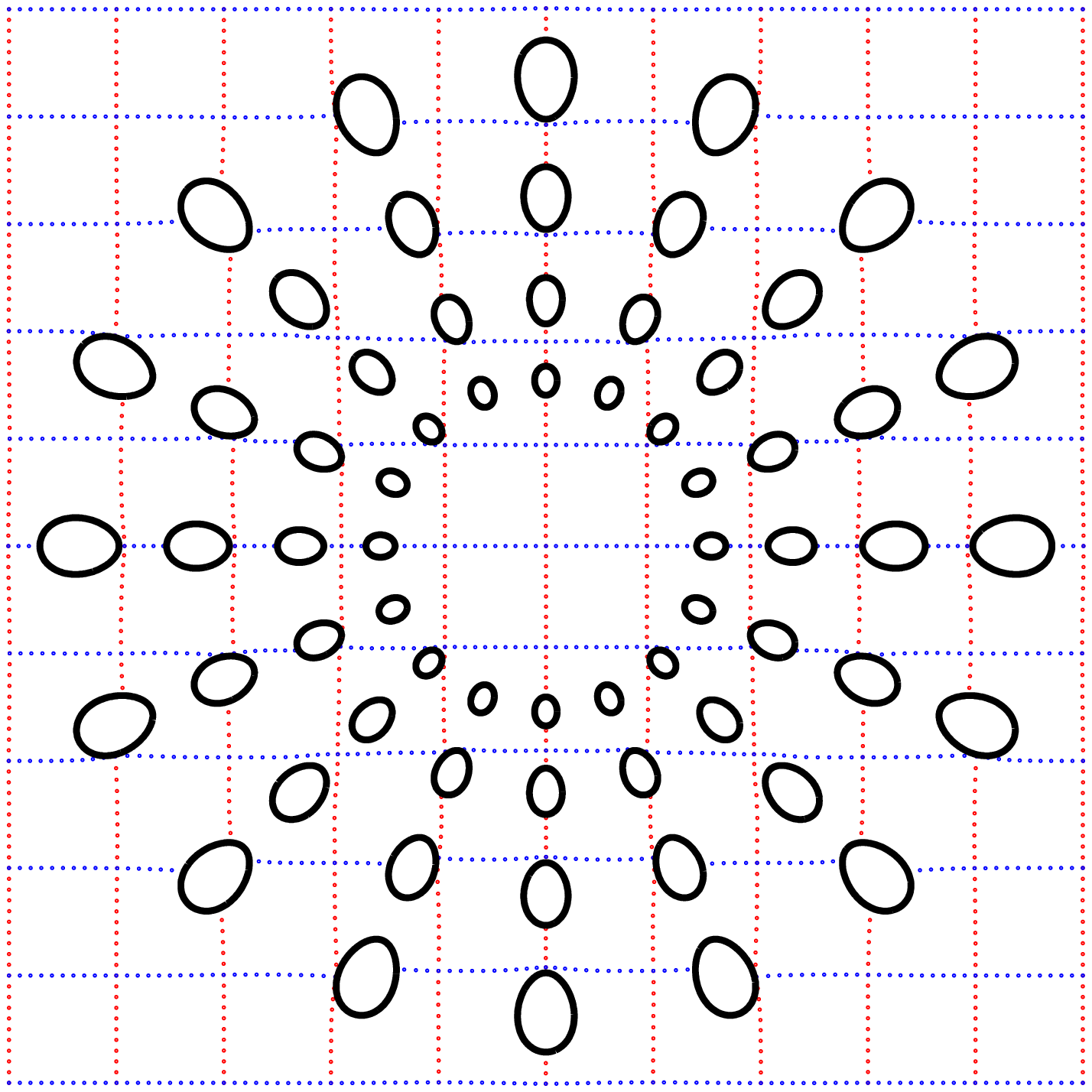}
}
\caption{Original domain for Example~\ref{ex:3} (left) and the 
 corresponding lemniscatic domain (right) obtained with $n=256$.}
\label{f:ex-3}
\end{figure}

\begin{figure}
\centerline{
\includegraphics[width=0.5\textwidth]{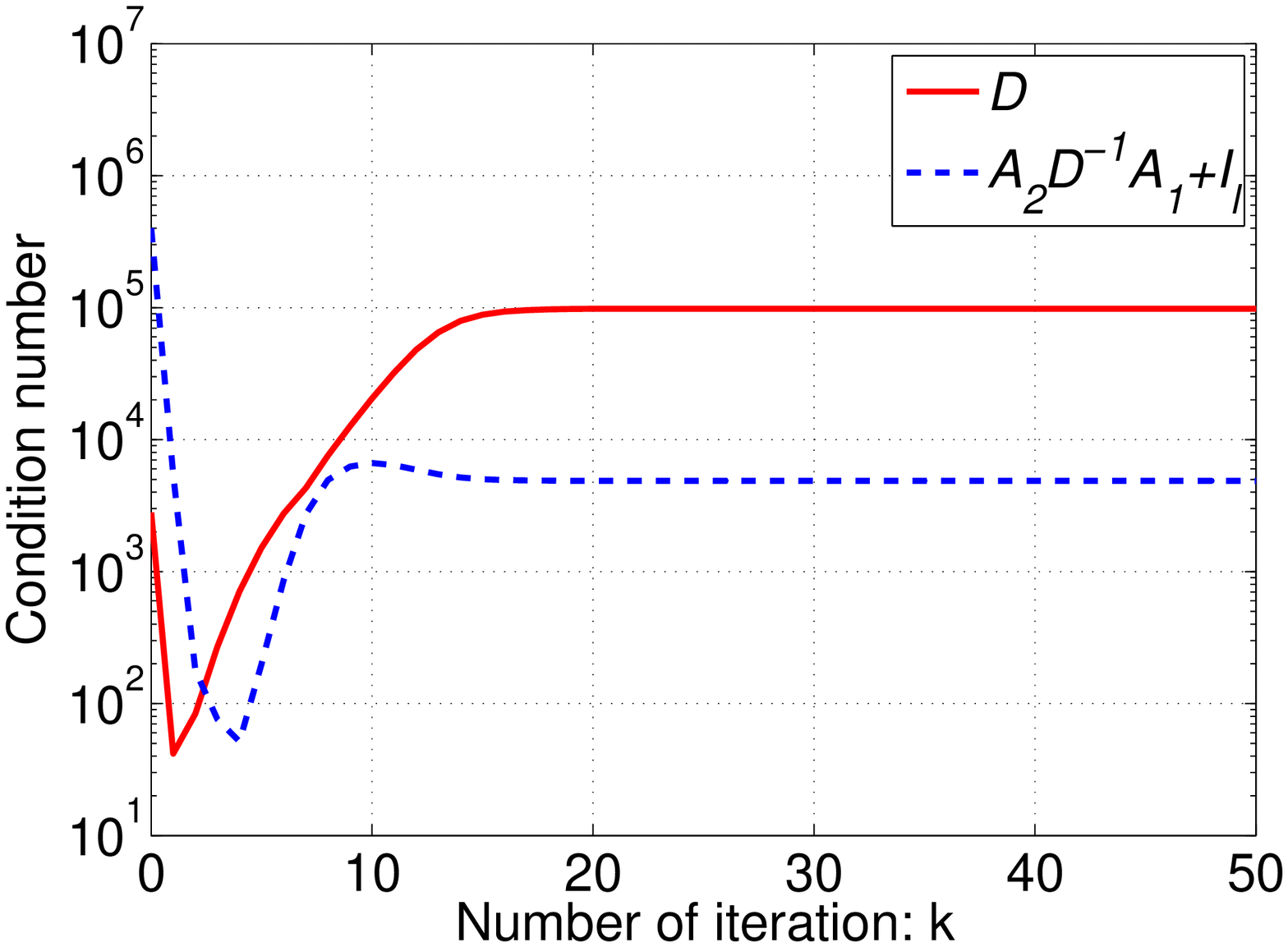}
\includegraphics[width=0.5\textwidth]{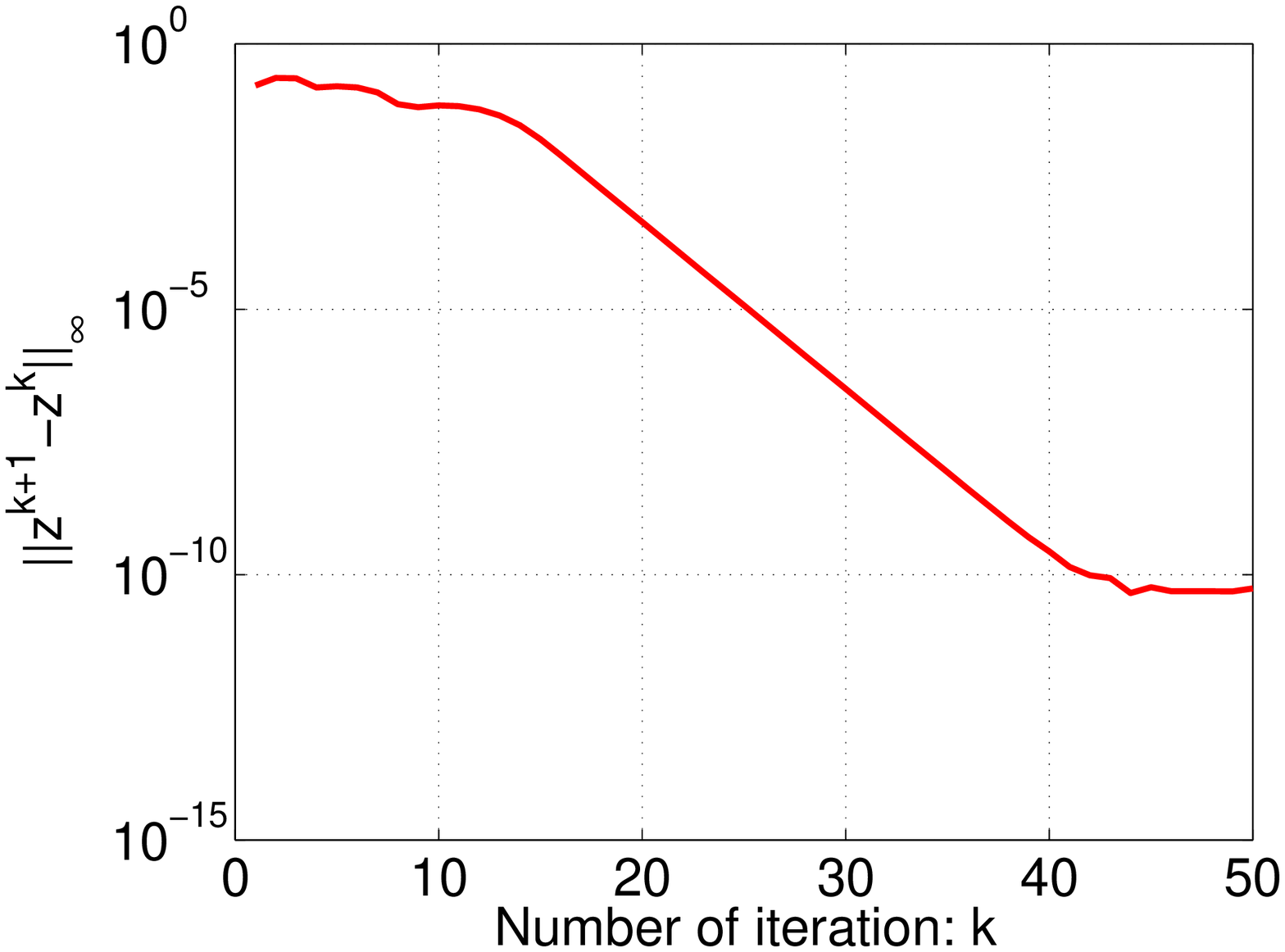}
}
\caption{2-norm condition numbers of $D$ and $A_2D^{-1}A_1+I_{\ell}$ (left) and 
the norms $\|\bz^{k+1}-\bz^{k}\|_\infty$ (right) in Example~\ref{ex:3}.}
\label{f:ex-3-err}
\end{figure}


\begin{example}\label{ex:4}
{\rm As indicated in Section~\ref{sect:compute_mv_tau}, our method can also 
be used when the boundary components of $\cK$ are only piecewise smooth 
Jordan curves. As an example we consider the unbounded domain $\cK$ 
exterior to four squares as shown on the left of Figure~\ref{f:ex-4}.
The computed lemniscatic domain obtained with $n=1024$ is shown on the 
right of that figure.
The GMRES method for the four linear algebraic systems required 26 iteration 
steps to attain a residual norm smaller than $10^{-14}$.
Figure~\ref{f:ex-4-err} shows the 2-norm condition numbers of $D$ and
$A_2D^{-1}A_1+I_{\ell}$ as well as the norms $\|\bz^{k+1}-\bz^{k}\|_\infty$
in the Newton iteration.
}\end{example}

\begin{figure}[tp] %
\centerline{
\includegraphics[width=0.45\textwidth]{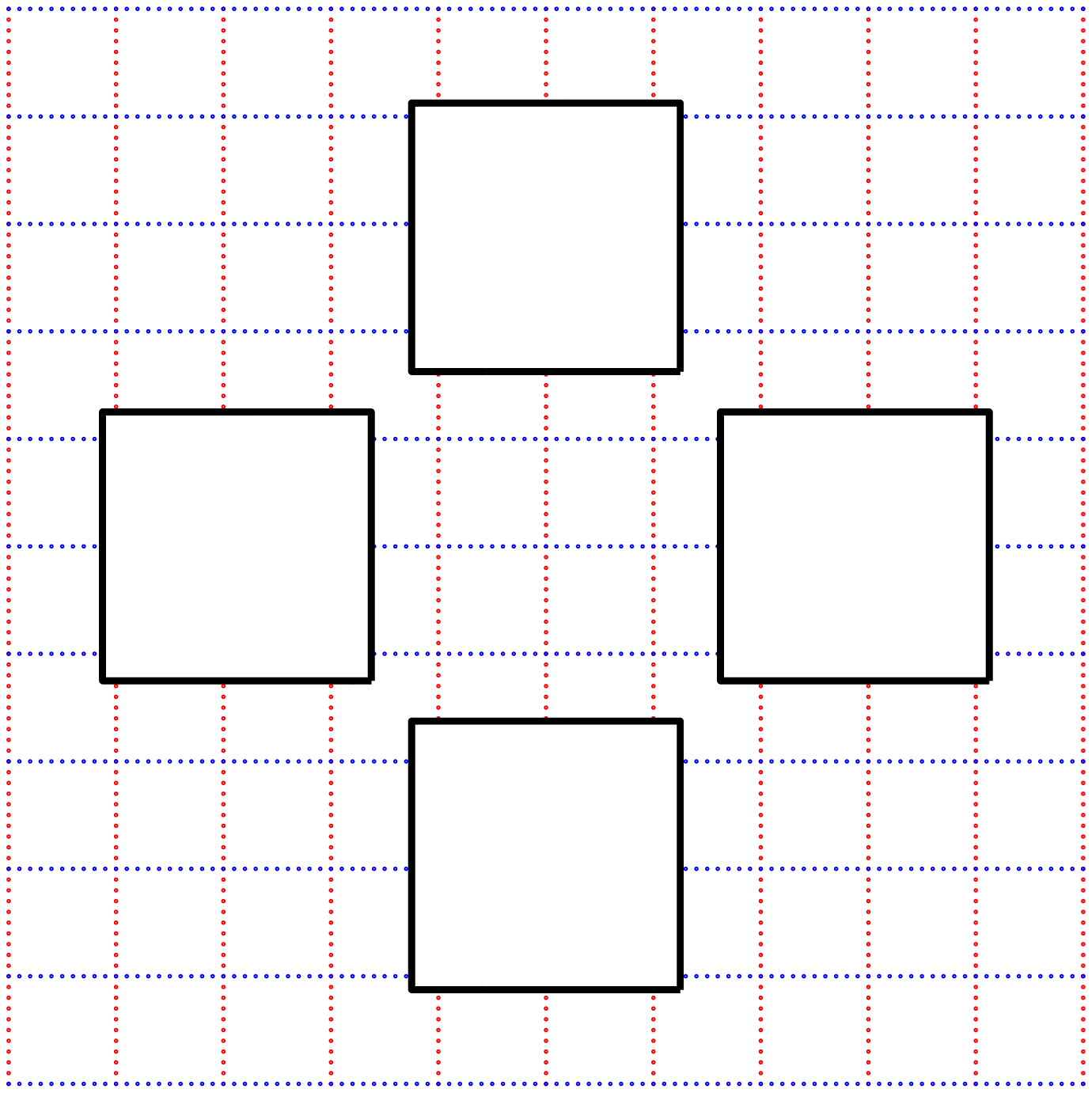}
\includegraphics[width=0.45\textwidth]{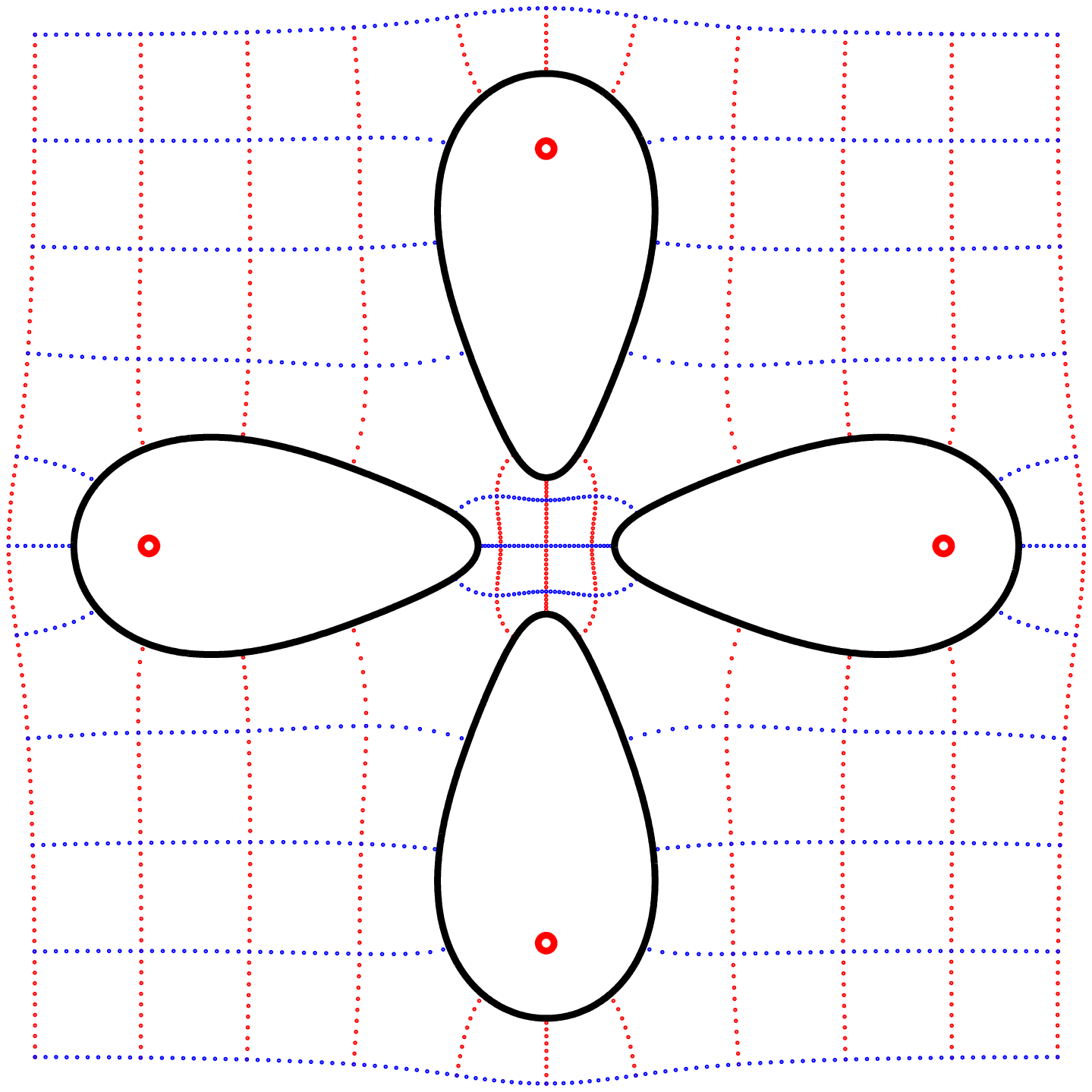}
}
\caption{Original domain for Example~\ref{ex:4} (left) and corresponding 
lemniscatic domain (right) obtained with $n=1024$.}
\label{f:ex-4}
\end{figure}

\begin{figure}[tp] %
\centerline{
\includegraphics[width=0.5\textwidth]{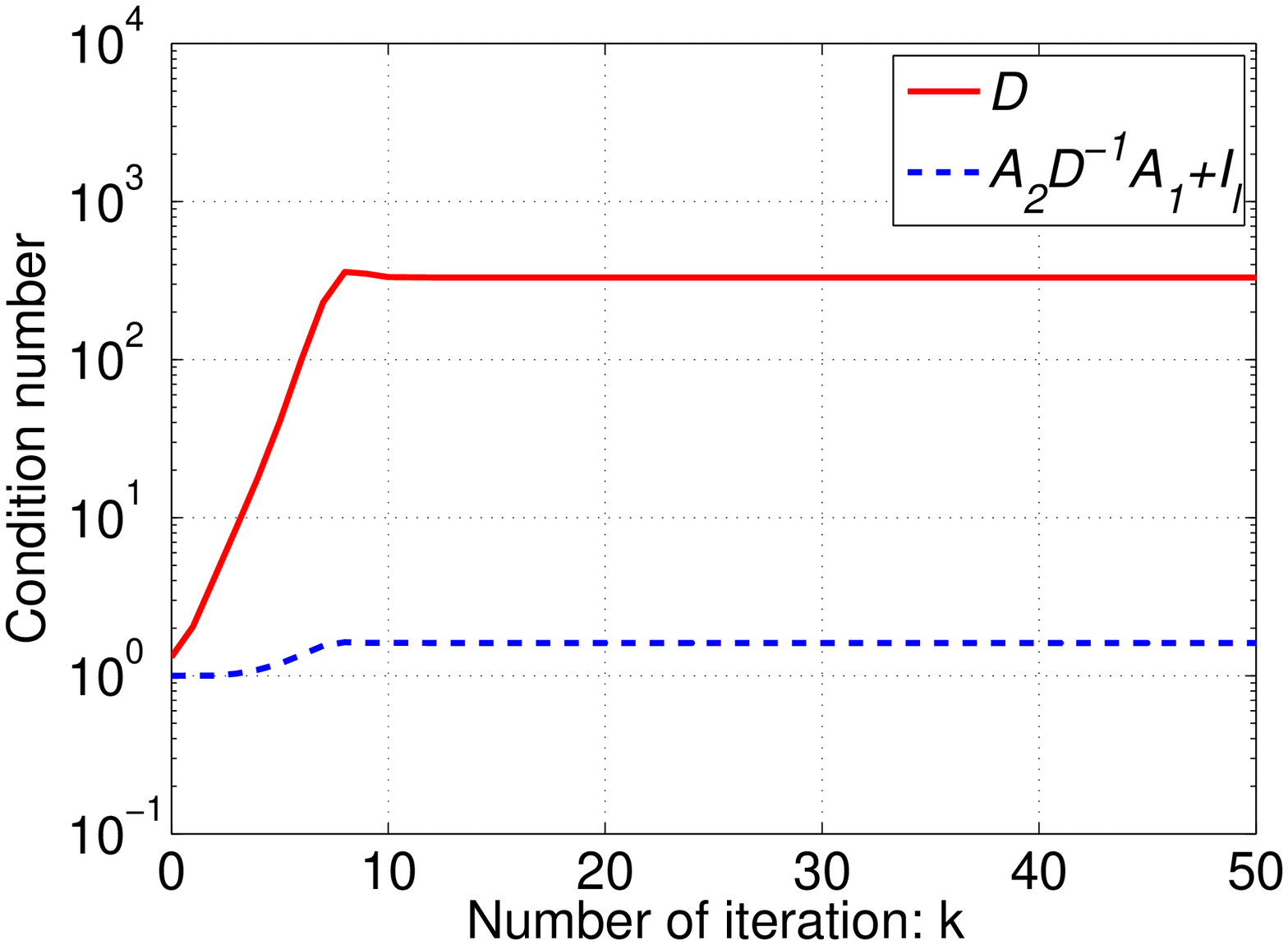}
\includegraphics[width=0.5\textwidth]{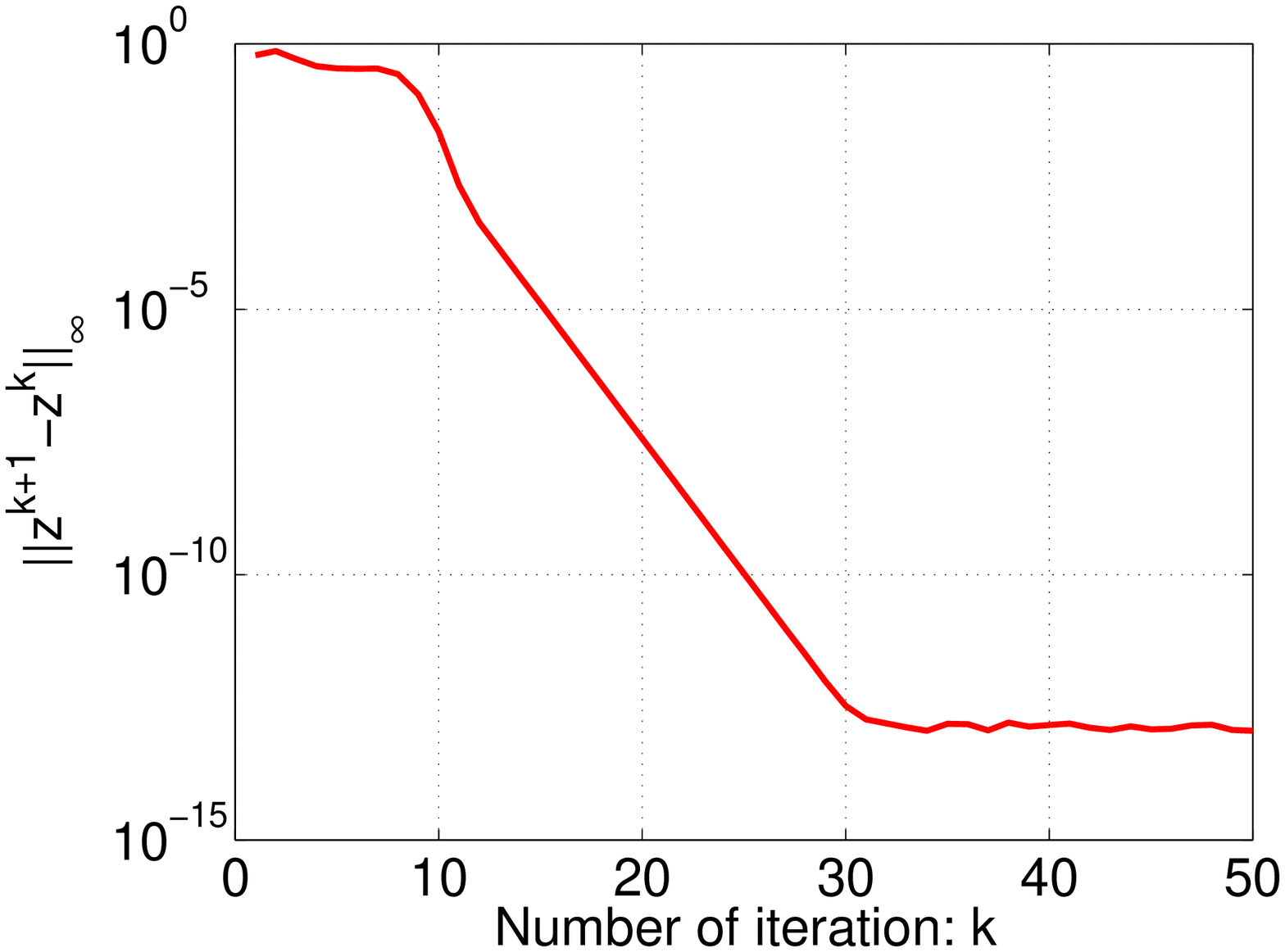}
}
\caption{2-norm condition numbers of $D$ and $A_2D^{-1}A_1+I_{\ell}$ (left) and 
the norms $\|\bz^{k+1}-\bz^{k}\|_\infty$ (right) in Example~\ref{ex:4}.}
\label{f:ex-4-err}
\end{figure}


\begin{example} \label{ex:3bw}
{\rm In this example we consider the unbounded domain $\cK$ exterior to three 
non-convex sets as shown on the left of Figure~\ref{fig:3bw}. The three sets 
are of the form introduced in~\cite{Koc-Lie2000}. The boundary curves are 
analytic and an analytic parameterization is known as well. The corresponding 
lemniscatic domain obtained with $n=1024$ nodes per boundary component is shown 
on the right of Figure~\ref{fig:3bw}. The GMRES method for the three linear 
algebraic systems required between $32$ and $34$ iteration steps to attain a 
residual norm smaller than $10^{-14}$.
Figure~\ref{fig:ex-3bw-err} shows the 2-norm condition numbers of $D$ and
$A_2D^{-1}A_1+I_{\ell}$ as well as the norms $\|\bz^{k+1}-\bz^{k}\|_\infty$
in the Newton iteration.
}\end{example}

\begin{figure}[tp] %
\centerline{
\includegraphics[width=0.5\textwidth]{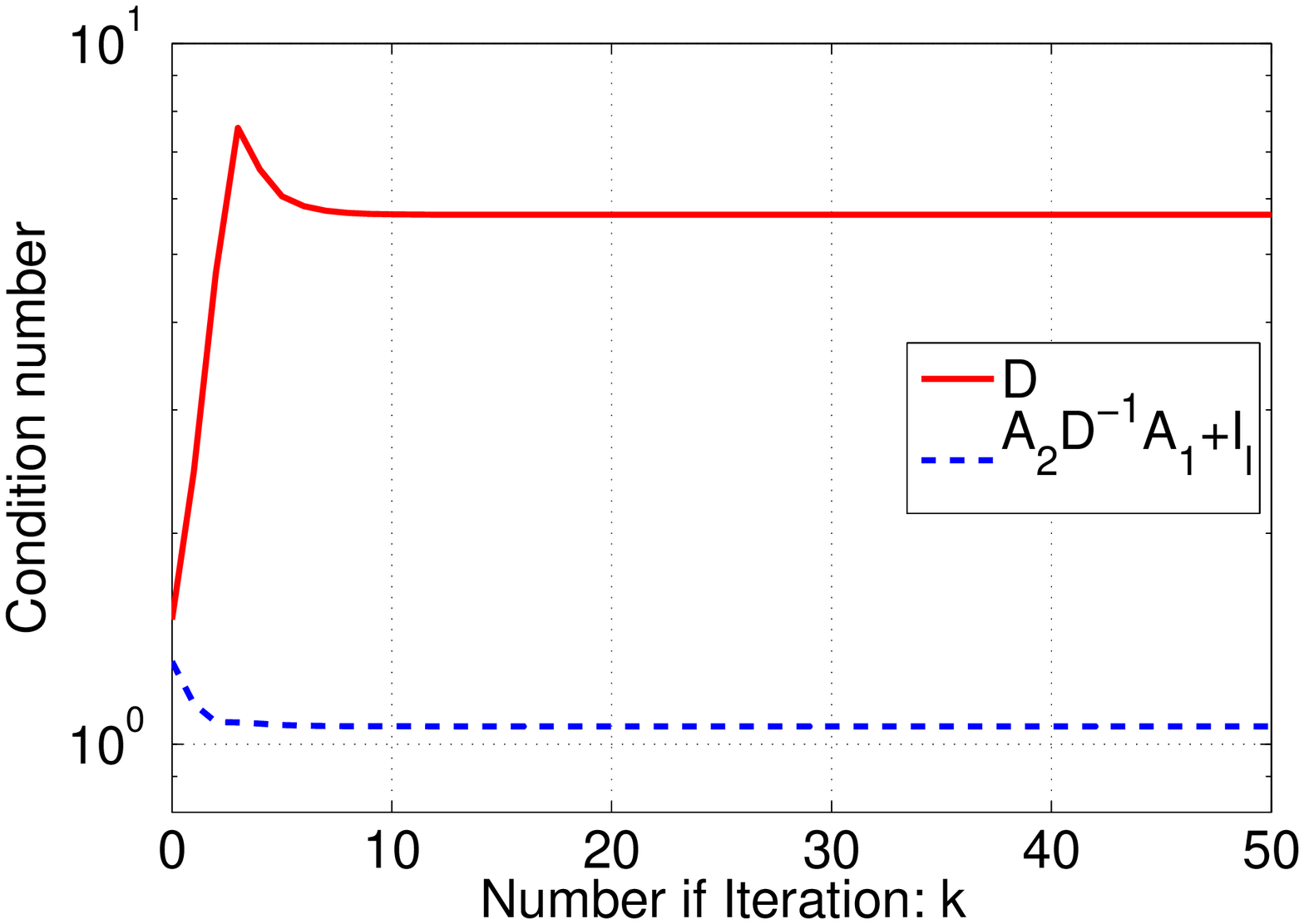}
\includegraphics[width=0.5\textwidth]{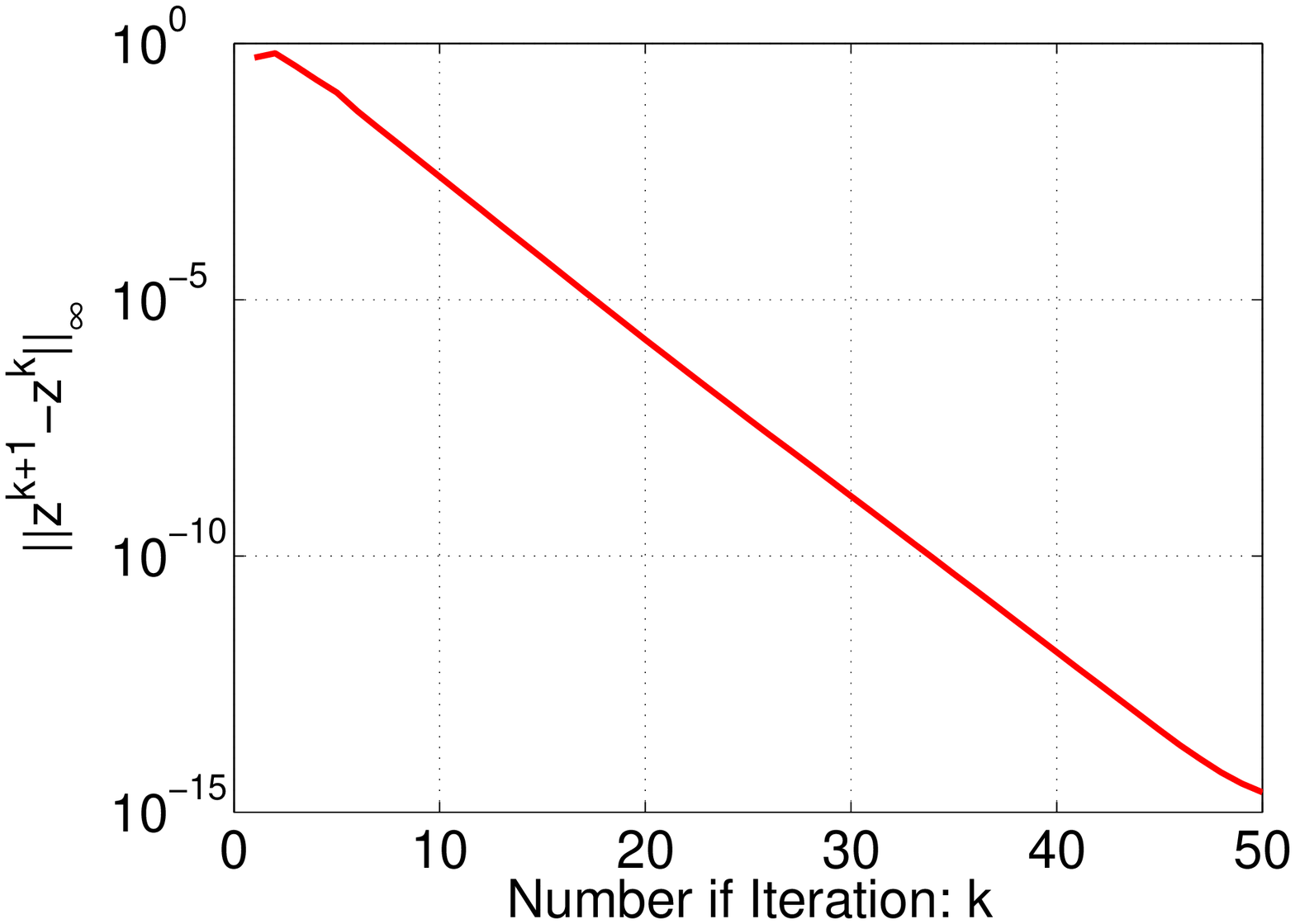}
}
\caption{2-norm condition numbers of $D$ and $A_2D^{-1}A_1+I_{\ell}$
(left) and the norms $\|\bz^{k+1}-\bz^{k}\|_\infty$ (right) 
in Example~\ref{ex:3bw}.}
\label{fig:ex-3bw-err}
\end{figure}

\section{Concluding remarks}
\label{sect:concl}

In this article we derived a method that numerically computes the conformal map 
from a given domain onto a lemniscatic domain.  The method relies on solving a 
boundary integral equation with the Neumann kernel.  It takes as input a 
parameterization of the boundary of the original domain and yields the 
parameters of the lemniscatic domain and the boundary values of the conformal 
map.  Using the numerically computed conformal map $\Phi : \cK \to \cL$ it is 
in particular possible to compute the Faber--Walsh polynomials associated with 
the compact set $\widehat{\C} \setminus \cK$.  The first numerical examples for 
such a computation are given in the paper~\cite{Set-Lie15_fwprop}.

The transfinite diameter (or logarithmic capacity) of the set
$\widehat{\C} \setminus \cK$ is one of the parameters in its corresponding 
lemniscatic domain $\cL$; see Theorem~\ref{thm:existence_lem_dom}.  This 
parameter is computed in the first step of the algorithm proposed in this 
paper; see Section~\ref{sect:compute_mv_tau}.  Since computing the transfinite 
diameter of compact sets is an interesting problem in its own right, we have 
derived a numerical method for this task, which is based on the method of 
Section~\ref{sect:compute_mv_tau}, in our paper~\cite{Lie-Set-Nas15_cap}.

Let us point out a few open questions.
As discussed in Section~\ref{sect:compute_mv_tau} and demonstrated numerically 
in Section~\ref{sect:examples}, the GMRES method converges very fast when 
solving the discretized boundary integral equation with the Neumann kernel.  A 
rigorous analysis of this effect is subject to further work.
Further, it would be interesting to analyze how the accuracy of the solution of 
the linear algebraic systems (solved with GMRES) affects the accuracy of the 
computed conformal map and of the parameters of the lemniscatic domain.
Finally, a ``black box'' starting point for the Newton iteration for solving 
the non-linear system~\eqref{eq:non-sys} would be of interest.

\paragraph*{Acknowledgements}
We thank Robert Luce for helpful discussions on Cauchy matrices and the 
solution of the systems \eqref{eqn:x}--\eqref{eqn:y_sys}.
We also thank Elias Wegert for suggesting our collaboration.
We further thank the anonymous referees for helpful comments.

\bibliographystyle{siam}
\bibliography{numconf}

\end{document}

%% file: def.tex
\newcommand{\C}{\mathbb{C}}

\newcommand{\cK}{\mathcal{K}}
\newcommand{\cL}{\mathcal{L}}

\newcommand{\abs}[1]{\vert #1 \vert}

\DeclareMathOperator{\re}{Re}
\DeclareMathOperator{\im}{Im}